\documentclass[12pt]{amsart}
\usepackage{inputenc, amsmath,amssymb,graphicx,verbatim, bbm, mathabx}
\usepackage[hidelinks]{hyperref}

\usepackage{amsthm}
\usepackage{tikz-cd, tikz, tikz-3dplot}
\usetikzlibrary{calc}
\usepackage{mathtools}
\usepackage{soul, scalerel}
\usepackage{multirow}
\usepackage[margin=1.25in]{geometry}

\newtheorem{theorem}{Theorem}
\newtheorem{lemma}[theorem]{Lemma}
\newtheorem{proposition}[theorem]{Proposition}
\newtheorem{corollary}[theorem]{Corollary}
\theoremstyle{definition}
\newtheorem{definition}[theorem]{Definition}

\theoremstyle{remark}
\newtheorem{remark}[theorem]{Remark}

\begin{document}
\title{Remarks on the Geometry of Sets in $\mathbb{Z}^d$ with Small Fourier $L^1$ Norm}

\author[Burgin]{Alex Burgin}
    \address{School of Mathematics, Georgia Institute of Technology, Atlanta GA 30332, USA}
    \email{aburgin6@gatech.edu}

\author[Mao]{Junzhe Mao}
    \address{School of Mathematics, Georgia Institute of Technology, Atlanta GA 30332, USA}
    \email{jmao87@gatech.edu}



\thanks{A.B. is supported in part by the Simons Foundation International through a Simons Dissertation Fellowship [SFI-MPS-SDF-00026777], the U.S. Department of Education through a GAANN fellowship, and the National Science Foundation [DMS-2247254]}

\begin{abstract}
    We study geometric inverse problems for finite sets $A\subset \mathbb{Z}^d$ whose Fourier $L^1$ norm, or Wiener norm, \begin{align*}
\|A\|:=\int_{\mathbb{T}^d}\Big|\sum_{a\in A}e(a\cdot t)\Big|dt
    \end{align*} is subpolynomial in $|A|$. In particular, we show a variety of geometric phenomena are incompatible with small Fourier $L^1$ norm. Our principal application concerns spherical Freiman models. Suppose $\|A\|=|A|^{o(1)}$ and $D\subseteq A$ is Freiman isomorphic to a set $S$ of lattice points on a sphere; then, $|D|/|A|$ must be polynomially small in $|A|$. Moreover, if $|D|/|A|$ is not too small, almost all of $S$ lies in a small number of spherical caps, and almost all caps have rich additive structure. The proof uses a new mechanism for studying additive structure across subsets of $A$, along with decoupling results of Bourgain--Demeter.
\end{abstract}
\maketitle 

\tableofcontents 

\section{Introduction and motivation}\label{sec:intro}

Let $f: \mathbb{Z}^d\rightarrow \mathbb{C}$ be a compactly supported function on the integer lattice, and define the Fourier $L^1$ norm (also known as the Wiener norm, or the Fourier algebra norm) of $f$ by \begin{align*}
    \|f\|:=\int_{\mathbb{T}^d}|\widehat{f}(\xi)|d\xi,\quad \widehat{f}(\xi):=\sum_{m\in \mathbb{Z}^d}f(m)e(m\cdot \xi).
\end{align*} We are mostly interested in the case where $f=\mathbf{1}_A$, the indicator function for a finite set on the integer lattice. In such a case, we simply write $\|A\|$ for the associated norm. Via Cauchy--Schwarz and Parseval's identity, it is clear that $\|A\|\leq |A|^{1/2}$. As for the lower bound, Littlewood \cite{hardyLittlewood} conjectured that for any finite set $A\subseteq \mathbb Z$, \begin{align}\label{eq:intro:1}
    \|A\|\gg \log |A|.
\end{align} 
This was proved independently by McGehee--Pigno--Smith \cite{MPS} and Konyagin \cite{konyaginLittlewood}. Note that \eqref{eq:intro:1} is sharp up to the implied constant by taking $A$ to be an arithmetic progression.
One can also consider lower bounds for sets in higher dimensions $d\geq 2$. It is clear that some sort of ``multidimensional structure" must be taken into account; by considering a product set $A=E_1\times \cdots \times E_d$, one observes that by applying (\ref{eq:intro:1}) to each coordinate one can obtain the better bound $\|A\|\gg \prod_{i=1}^d \log |E_i|$. Whereas if a set in $\mathbb{Z}^d$ is zero on all coordinates other than the first, for example, the bound (\ref{eq:intro:1}) is sharp. One may view the former as being a ``genuinely $d$-dimensional" set, while the latter seems more one-dimensional; this idea was explored by Petridis \cite{Giorgis_2013} and Hanson \cite{Hanson_2021}. They provided lower bounds for $\|A\|$ in terms of rows (intersections of $A$ with lines) and slices (intersections of $A$ with planes).

In light of these lower bounds, one may ask what the extremizers are: that is, which sets $A$ realize a smaller-than-expected value of $\|A\|$? As mentioned above, an arithmetic progression realizes the bound (\ref{eq:intro:1}). Moreover, by H\"older's inequality and a bit of interpolation, a small value of $\|A\|$ forces $A$ to have a large additive energy; here we recall the additive energy of $A$ is defined as 
\begin{align*}
    E_2(A):=\#\{(a_1,a_2,a_3,a_4)\in A^4:a_1+a_2=a_3+a_4\}.
\end{align*} 
This inverse problem was studied recently by Bloom and Green \cite{Bloom-Green_2025}, who showed that, if $\|A\|\leq K\log |A|$, then $A$ contains arbitrarily long arithmetic progressions as $|A|\rightarrow\infty$, for fixed $K>0$. Some related results were used by Bedert \cite{bedert2025largesumfreesubsetssets}, which led to a breakthrough in sum-free subsets.

We are interested in a larger class of sets than those affected by these previous inverse results. In particular, we show that for $A\subset \mathbb{Z}^d$ with the condition $ \|A\|=|A|^{o(1)}$, one may still discern interesting geometric information about $A$.

High-dimensional geometry has played an important role in arithmetic combinatorics for many years. Of particular note is the work of Behrend \cite{behrend}, who used the geometry of the sphere in high dimensions to provide a construction of a set $B\subset [N]$ that avoids all nontrivial 3-term arithmetic progressions (is `3-AP free'), yet still satisfies $|B|\geq \frac{N}{\exp(C\log^{1/2}N)}$ for some constant $C>0$. This set takes the form $B=\{\sum_{i=0}^{n-1}a_i(2b)^i:a_i\in \{0,1,\ldots,b-1\},\sum_{i=0}^{n-1}a_i^2=k\}$ for parameters $n, b, k$ depending on $N$, which is a projection of the sphere into the integers. Since the sphere is 3-AP free, and the projection is Freiman-isomorphic\footnote{Two sets are Freiman-isomorphic if they share similar additive structure: see Definition \ref{def:Freiman}.} to the sphere, one obtains that $B$ is 3-AP free. On the flip side, Freiman's theorem uses the image of a box or a centered convex body to characterize sets with rich additive structure. We refer to~\cite{TaoVu_2009} for more examples. These examples emphasize that the geometry of a high-dimensional additive model--be it curved, convex, or flat--encodes substantial information about the underlying set.

Let us highlight some classical results connecting geometry and the Fourier $L^1$ norm. For a set $A\subset \mathbb{R}^d$ and a scaling factor $R>0$, let $RA$ denote the scaled set \begin{align*}
    RA:=\{x\in \mathbb{R}^d:x/R\in A\}
\end{align*} and let $A_R$ denote the set of lattice points inside $RA$, that is, \begin{align*}
    A_R:=\{y\in \mathbb{Z}^d:y/R\in A\}.
\end{align*} Regarding centrally symmetric convex sets, Yudin \cite{yudin_1975} showed the following.

\begin{theorem}[Yudin, \cite{yudin_1975}]\label{thm:yudin}
    Let $d\geq 2$ and $K\subset \mathbb{R}^d$ be a centrally symmetric convex body with bounded Minkowski surface area. Then, \begin{align*}
        \|K_R\|\ll_{d,K} R^{\frac{d-1}{2}}.
    \end{align*}
\end{theorem} In particular, since $|K_R|\asymp R^d$, this gives an upper bound on $\|K_R\|$ that is a power-savings over the trivial bound $R^{\frac{d}{2}}$. Specializing to polyhedra, one may get an essentially optimal result, due to Belinsky~\cite{belinsky_1977}:

\begin{theorem}[Belinsky, \cite{belinsky_1977}]
    For any bounded, convex $d$-dimensional polyhedron $W\subset \mathbb R^d$, one has that \begin{align*}
        \|W_R\|\asymp_{d,W} \log^d R.
    \end{align*}
\end{theorem} However, the upper bound given by Yudin is sharp, as can be seen with the so-called spherical Dirichlet kernel.

\begin{theorem}
    Let $B$ be the unit ball in $\mathbb{R}^d$, $d\geq 2$. Then, \begin{align*}
        \|{B_R}\|\asymp_d R^{\frac{d-1}{2}}.
    \end{align*}
\end{theorem}
This is a classical result in the study of partial sums of Fourier series. We refer to \cite{IosevichLiflyand_2014} for a proof and more discussions on related topics.
Comparing this result with the normal $d$-dimensional Dirichlet kernel, which is the Fourier transform of the lattice points in a dilated cube, it is clear that the geometry of the underlying set plays a decisive role in the asymptotics of the associated Fourier $L^1$ norm.

\subsection{Main results}

We are primarily interested in the ``small norm" regime where $\|A\|=|A|^{o(1)}$; in many cases, this gives enough control to be able to deduce detailed geometric information about the sets in question.

We specialize to spherical models, at least for the next few theorems, for a few reasons: not only is the geometry simple to understand, but there are also strong decoupling results available for appropriate partitions of a spherical lattice, and one has regularity in lattice-point counts (at least, for sufficiently high dimensions). This should not be interpreted as saying that the theory is specific only to spheres; our framework is actually quite general (we discuss further applications in \S\ref{sec:extensionsLimitations}). So, to this end, for $R>0$ let $\mathcal{F}_{d,R^2}:=\{x\in \mathbb{Z}^d:|x|^2=R^2\}$ be the set of lattice points on the sphere of radius $R$ in $\mathbb R^d$. Throughout this paper, we will assume that $R^2$ is a positive integer so that $\mathcal{F}_{d,R^2}$ is nonempty.

To pose our theorems, we require the concept of a Freiman isomorphism. Two sets are Freiman-isomorphic if they share similar additive structure (for example, one may take a suitable projection of a high-dimensional set into one dimension).

\begin{definition}[Freiman $h$-homomorphism]\label{def:Freiman} Let $D\subset G$ and $D'\subset G'$, where $G,G'$ are abelian groups. A map $\phi:D\rightarrow D'$ is a Freiman $h$-homomorphism if \begin{align*}
    d_1+\cdots +d_h=d_{h+1}+\cdot +d_{2h}\implies \phi(d_1)+\cdots +\phi(d_h)=\phi(d_{h+1})+\cdots +\phi(d_{2h})
\end{align*} for all $d_i\in D$. A Freiman $h$-isomorphism is a bijection in which both directions hold.
\end{definition}

\begin{definition}
    For a set $A\subset \mathbb{Z}^m$, we say that a nonempty subset $D\subset A$ is a \textit{d-dimensional spherical Freiman model} if it is Freiman 2-isomorphic to a subset $S$ of the lattice sphere $\mathcal{F}_{d,R^2}$, in dimension $d\geq 2$.
\end{definition}

If a finite set $A$ admits a spherical Freiman model $D$ as a subset, write \begin{align*}
    \alpha:=\frac{|D|}{|A|},\quad \beta:=\frac{|S|}{|\mathcal{F}_{d,R^2}|}.
\end{align*}

Our first theorem shows that large spherical models inside $A$ are incompatible with a small Fourier $L^1$ norm.

\begin{theorem}\label{thm:iteratedDimensionReduction}
    Let $A\subset \mathbb{Z}^m$ be finite and nonempty. Suppose that $A$ contains a $d$-dimensional spherical Freiman model $D$, in dimension $d\geq 3$. Then, for any $\epsilon>0$ \begin{align*}
        \|A\|\gg_{d,\epsilon} \alpha^{\frac{1}{2}+\frac{7}{2t_d}-\epsilon}|A|^{\frac{7}{2t_d}-\epsilon},\quad t_d:=2^{d+1}-7.
    \end{align*} 
\end{theorem}

The theorem has the following inverse formulation.

\begin{corollary}[Large spherical models are forbidden by subpolynomial norm] Under the hypotheses of Theorem \ref{thm:iteratedDimensionReduction}, if $\|A\|=|A|^{o(1)}$, then \begin{align*}
    |D|\ll_{d,\epsilon} |A|^{1-\frac{7}{2^{d+1}}+o(1)}.
\end{align*}
    
\end{corollary}

One may also get a result with potentially sharper bounds if $S$ encompasses more of the sphere, or if $D$ is an even sparser subset of $A$; as a trade-off, one has less local information about the distribution of the points.

\begin{theorem}\label{thm:global}
    Let $d\geq 5$ and $A\subset \mathbb{Z}^m$ be finite and nonempty. Suppose that $A$ contains a $d$-dimensional spherical Freiman model $D$. Then, for every $\epsilon>0$, 
    \begin{align*}
        \|A\|\gg_{d,\epsilon} |A|^{\frac{1}{d-2}-\epsilon}\alpha^{\frac{d}{2(d-2)}-\epsilon}\beta^{\frac{d-4}{2(d-2)}+\epsilon}.
    \end{align*}
\end{theorem}

\begin{corollary}
    Under the hypotheses of Theorem \ref{thm:global}, if $\|A\|=|A|^{o(1)}$, then \begin{align*}
        \alpha^d\beta^{d-4}\ll_{d,\epsilon} |A|^{-2+o(1)}.
    \end{align*} In particular, either $D$ is a very sparse subset of $A$, or $S$ is a very sparse subset of $\mathcal{F}_{d,R^2}$.
\end{corollary}

Besides these density results, one may also deduce finer details about the distribution of $S$ on $\mathcal{F}_{d,R^2}$. More precisely,
for $R^{-1}\leq \sigma\leq c_d$, we partition $\mathcal{F}_{d,R^2}$ into $\asymp_d \sigma^{-(d-1)}$ many standard $\sigma$-caps, meaning the lattice partition induced by the usual tangential $\sigma$-scale decomposition of the unit sphere (for more detailed information, see Appendix A). For $S\subseteq \mathcal{F}_{d,R^2}$, write $S_\theta$ for the intersection of $S$ with the cap indexed by $\theta$. This partition is in a form amenable to the decoupling results of \cite{BourgainDemeter2015}. Our next result shows that a large portion of $S$ can be covered by a small number of $\sigma$-caps if $\|A\|$ is small, and also that this set concentrates in a small number of spherical sections.

\begin{definition}
    A \textit{proper spherical section} of the lattice sphere set $\mathcal{F}_{d,R^2}$ is a nonempty intersection $P\cap \mathcal{F}_{d,R^2}$, where $P\subset \mathbb{R}^d$ is a proper affine subspace.
\end{definition}

\begin{theorem}[99 Percent Concentration and Directional Compression]\label{thm:mainThm2}
    Let $d\geq 5$ and $A\subset \mathbb{Z}^m$ be finite and nonempty. Suppose that $A$ contains a $d$-dimensional spherical Freiman model $D$, and that $|D|$ is sufficiently large in terms of $d$. Partition $\mathcal{F}_{d,R^2}$ into standard $\sigma$-caps, where $\sigma=|S|^{-1/(d-1)}$; this induces a partition of $S$, so that $S=\bigsqcup_{\theta\in\Theta}S_\theta$. Then for sufficiently small $\epsilon>0$, there exists a set $S_0\subset S$ with $|S\setminus S_0|\leq \frac{|S|}{100}$ and satisfying both of the following: \begin{enumerate}
        \item[(1)] $S_0$ is covered by \begin{align*}
            O_{d,\epsilon}(\alpha^{-\frac{2}{d-1}+\epsilon}|A|^{1-\frac{2}{d-1}+\epsilon}\|A\|^2)
        \end{align*} proper spherical sections, and
        \item[(2)] $S_0$ lies in a \begin{align*}
            O_{d,\epsilon}\Big(|A|^{-\frac{1}{d-1}+\epsilon}\alpha^{-\frac{1}{2}-\frac{1}{d-1}-\epsilon}\|A\|\Big)^{2+\Omega(1/t_d)}
        \end{align*} proportion of the total number of caps (here $t_d$ is as in Theorem \ref{thm:iteratedDimensionReduction}).
    \end{enumerate}
\end{theorem}

A key point is that, with the partition of the spherical model, one not only has concentration into geometric subspaces, but also hereditary additive structure on almost all cap mass.

\begin{theorem}[Hereditary structure on almost all cap mass]\label{thm:mainThm3} Under the same hypotheses as Theorem \ref{thm:mainThm2}, partition $\mathcal{F}_{d,R^2}$ into standard $\sigma$-caps, where $R^{-1}\leq \sigma\ll_d 1$. Let $S=\bigsqcup_{\theta\in \Theta}S_\theta$ be the induced partition. Then, for every $\epsilon>0$ and $\rho\in (0,1)$ there exists a subset of indices $\Theta_{good}\subseteq \Theta$ such that \begin{align*}
    \sum_{\theta\in \Theta_{good}}|S_\theta|\geq (1-\rho)|S|
\end{align*} and the following holds for every $\theta\in \Theta_{good}$. If $\varnothing\neq Y\subseteq S_\theta$ and $\delta= |Y|/|S_\theta|$, then some nonzero $z\in Y+Y$ satisfies \begin{align*}
    |Y\cap (z-Y)|\gg_{d,\epsilon}\sigma^{d-3+\epsilon}\frac{\delta \rho \alpha }{\|A\|^2}|Y|.
\end{align*}
\end{theorem}

Together, these theorems show that a large high-dimensional spherical model cannot exist inside a set $A$ of small Fourier $L^1$ norm: it must either be small or low-dimensional, and when it does capture some spherical geometry, these parts concentrate into a small number of highly-structured caps.\\

Theorems \ref{thm:mainThm2} and \ref{thm:mainThm3} are special cases of a more general phenomenon concerning the flat-vs-curved dichotomy. Let $\mathcal{P}(A)$ denote the set of probability measures supported on $A\subset \mathbb{Z}^d$. For each $\mu\in \mathcal{P}(A)$, one may form its weighted additive energy $E_2(\mu):=\sum_z (\mu*\mu)(z)^2$; one can then define the \textit{additive capacity} of $A$ by \begin{align*}
    \text{Cap}_2(A):=\Big(\inf_{\mu\in \mathcal{P}(A)}E_2(\mu)\Big)^{-1}.
\end{align*} We study additive capacity in more detail in Sections \ref{sec:additiveCapacity} and \ref{sec:partitionDecoupling}, but note here that this notion is very useful in combination with decoupling results. If $S=\bigsqcup_\theta S_\theta$ is a partition, we say that the partition has $L^4$ decoupling constant $K $ if \begin{align}\label{eq:intro:decoupling}
    \Big\|\sum_{x\in S}a_xe(x\cdot \xi)\Big\|_{L^4(\mathbb{T}^d,\xi)}\leq K\Big(\sum_\theta \Big\|\sum_{x\in S_\theta}a_xe(x\cdot \xi)\Big\|_{L^4(\mathbb{T}^d,\xi)}^2\Big)^{1/2},\quad (a_x)\in \mathbb{C}^S.
\end{align} Such an inequality holds for many curved surfaces; see \cite{bourgainDemeter2017}. Our next result shows that the Fourier $L^1$ norm of $A$ controls additive capacity over a partition if it admits such a decoupling.

\begin{theorem}[Fourier $L^1$ norm controls partition capacity]\label{thm:masterThm}
    Let $A\subset \mathbb{Z}^m$ be finite, and suppose $D\subseteq A$ is Freiman 2-isomorphic to $S\subset \mathbb{Z}^d$. If $S=\bigsqcup_{\theta\in \Theta}S_\theta$ satisfies (\ref{eq:intro:decoupling}), and suppose $S_\theta\neq \emptyset$ for all $\theta\in \Theta$, then 
    \begin{align*}
        \sum_{\theta\in \Theta} \mathrm{Cap}_2(S_\theta)^{1/2}\leq K^2|A|^{1/2}\|A\|.
    \end{align*}
\end{theorem}

These estimates allow one to move from the global structure of $A$ to a collection of local conclusions. Existing arguments of Shkredov \cite{shkredov} and Bedert \cite{bedert2025largesumfreesubsetssets} give global or hereditary energy lower bounds for sets with small Fourier $L^1$ norm, but the key point here is that one can get quantitatively effective results for the local distribution and geometry of these sets.\\

The preceding results on spherical models have the following simpler, but particularly concrete, counterpart concerning convexly independent sets in $\mathbb{Z}^d$. Recall that a convexly independent set $A\subset \mathbb{Z}^d$ satisfies, for all $a\in A$, that $a$ is not in the convex hull of $A\setminus \{a\}$. We show that any set with a large convexly independent subset must have a nontrivial Fourier $L^1$ norm, by considering the inclusion of 3-term arithmetic progressions.

\begin{theorem}\label{thm:3-AP}
    There exists an absolute constant $c\in (0,1)$ such that the following holds. Suppose that $A\subset \mathbb{Z}^d$ is sufficiently large, and satisfies $\|A\|\leq \exp(\log^{c}|A|)$. Then, any subset $A'\subset A$ with $|A'|\geq \frac{|A|}{\|A\|}$ contains a nontrivial 3-term arithmetic progression.
\end{theorem}

In particular, no such subset $A'$ admits a Freiman 2-isomorphic image that is convexly independent. Thus, a set with sufficiently small Fourier $L^1$ norm cannot contain a large subset with any image of a convexly independent set. The conclusion is particularly illustrative in light of Behrend's construction. Behrend obtains large progression-free sets in $\mathbb{Z}$ via projecting a subset of a high-dimensional sphere into the integers, where the strict convexity of the sphere precludes any 3-term arithmetic progressions.

\subsection{Organization of the paper}

We begin in \S\ref{sec:projections} by discussing how the Fourier $L^1$ norm behaves through Freiman-preserving projections (that is, maps $\phi:\mathbb{Z}^m\rightarrow \mathbb{Z}^n$ that preserve additive relations). This is a ``robustness" test that demonstrates that the norm is actually detecting additive structure (in particular, we show that for many types of Freiman homomorphisms, a small norm of the preimage induces a small norm of the image). We next provide a short proof and commentary around Theorem \ref{thm:3-AP} in \S\ref{sec:3APs}.

Sections \ref{sec:additiveCapacity}, \ref{sec:partitionDecoupling}, and \ref{sec:mainProofs} are the main technical ingredients for the proof of  Theorems  \ref{thm:mainThm2}, \ref{thm:mainThm3}, and \ref{thm:masterThm}. In \S\ref{sec:additiveCapacity}, we discuss additive capacity and its properties, and its relation to Freiman homomorphisms. In \S\ref{sec:partitionDecoupling}, we prove that additive capacity is super-additive with respect to a partition that decouples; this leads to the proof of Theorem \ref{thm:masterThm}, which is the abstract input to Theorems  \ref{thm:mainThm2} and \ref{thm:mainThm3}. In \S\ref{sec:mainProofs}, we first prove Theorem \ref{thm:iteratedDimensionReduction} and then use known decoupling results to apply Theorem \ref{thm:masterThm} to the sphere to obtain Theorems \ref{thm:mainThm2}, \ref{thm:mainThm3}.

We end with \S\ref{sec:extensionsLimitations}, where we discuss further applications of this general theory; in particular, the decoupling framework applies not only to spheres, but to wider classes of curved surfaces.

\section{Freiman-preserving projections}\label{sec:projections}

We begin by discussing why sets in $\mathbb{Z}^d$ are relevant to additive combinatorics in $\mathbb{Z}$, with a view towards Fourier $L^1$ norm. This relevance is twofold: sets $S\subset \mathbb{Z}^d$ can be projected into $\mathbb{Z}$, preserving both additive structure and small Fourier $L^1$ norm, and sets $S'\subset \mathbb{Z}$ may have high-dimensional structure that deeply influences this norm.

The following lemma shows that one can embed a higher-dimensional set in $\mathbb{Z}$ while preserving small Fourier $L^1$ norm.

\begin{lemma}\label{lem:projection}
    Suppose $W\subset \{-d,\cdots ,d\}^n$ is finite, and for $b\in \mathbb{N}$ with $b>2d$ define \begin{align*}
        A:=\Big\{\sum_{i=1}^{n}w_ib^{i-1}:w\in W\Big\}.
    \end{align*} Then, \begin{align*}
        \|A\|\leq (2+2\log b)^n\|W\|.
    \end{align*}
\end{lemma}

\begin{remark} Given an integer $k\geq 2$, if we choose $b>2kd$, then $A$ is Freiman $k$-isomorphic to $W$ in the natural way. This shows we can control the Fourier $L^1$ norm of a lower-dimensional projection effected through a base encoding in terms of the Fourier $L^1$ norm of the higher-dimensional set. 
\end{remark}

\begin{proof}
    From periodicity, we may compute 
    \begin{align*}
    \|A\| &= b^{-n}\int_\mathbb T \sum_{\ell = 0}^{b^n-1}|\widehat{\mathbf 1_A}(t+\ell/b^n)| dt\\&\leq b^{-n}\sup_t\sum_{0\leq \ell<b^n}\Big|\sum_{r\in \{-d,...,d\}^n}\mathbf{1}_W(r)\prod_{i=0}^{n-1}e(r_ib^i(t+\ell/b^n))\Big|.
\end{align*} By Fourier inversion, we may write \begin{align}
    \mathbf{1}_W(x)=\int_{\mathbb{T}^n}\widehat{\mathbf{1}_W}(v)e(-v\cdot x)dv.
\end{align} 
    This provides then that \begin{align*}
    &\sum_{r\in \{-d,...,d\}^n}\mathbf{1}_W(r)\prod_{i=0}^{n-1}e(r_ib^i(t+\ell/b^n))\\ &= \sum_{r\in \{-d,...,d\}^n}\int_{\mathbb{T}^n}\widehat{\mathbf{1}_W}(v)e(-v\cdot r)dv\prod_{i=0}^{n-1}e(r_ib^i(t+\ell/b^n)) \\ &= \int_{\mathbb{T}^n}\widehat{\mathbf{1}_W}(v)\sum_{r\in \{-d,...,d\}^n}\prod_{i=0}^{n-1}e(r_i(-v_{i+1}+b^i(t+\ell/b^n)))dv \\ &= \int_{\mathbb{T}^n}\widehat{\mathbf{1}_W}(v)\prod_{i=0}^{n-1}\sum_{-d\leq k\leq d}e(k(-v_{i+1}+b^i(t+\ell/b^n)))dv
\end{align*} and so, by the triangle inequality, this has magnitude at most \begin{align*}
    \int_{\mathbb{T}^n}|\widehat{\mathbf{1}_W}(v)|\prod_{i=0}^{n-1}\Big|\sum_{-d\leq k\leq d}e(k(-v_{i+1}+b^i(t+\ell/b^n)))\Big|dv.
\end{align*} Summing over $\ell$ and interchanging summations, and bringing in the supremum, then provides that \begin{align*}
    \|A\|\leq b^{-n}\int_{\mathbb{T}^n}|\widehat{\mathbf{1}_W}(v)|\sup_t\sum_{0\leq \ell<b^n}\prod_{i=0}^{n-1}\Big|\sum_{-d\leq k\leq d}e(k(-v_{i+1}+b^i(t+\ell/b^n)))\Big|dv.
\end{align*} The interior sum is bounded above by $\min\{2d+1,\frac{1}{2\|-v_{i+1}+b^i(t+\ell/b^n)\|_\mathbb{T}}\}$, and so \begin{align*}
    \|A\|\leq b^{-n}\int_{\mathbb{T}^n}|\widehat{\mathbf{1}_W}(v)|\sup_t\sum_{0\leq \ell<b^n}\prod_{i=0}^{n-1}\min\Big\{2d+1,\frac{1}{2\|-v_{i+1}+b^i(t+\ell/b^n)\|_\mathbb{T}}\Big\}dv.
\end{align*} Now, writing $t=\sum_{i=1}^n\frac{t_i}{b^i}+\epsilon$, where $t_i\in \{0,..,b-1\}$ and $0\leq \epsilon<b^{-n}$, one has that \begin{align*}
    b^it\equiv \frac{t_{i+1}}{b}+\epsilon'\ (\text{mod }1),
\end{align*} where $0\leq \epsilon'<1/b$. Then, on $\mathbb{T}$, the quantity $-v_{i+1}+b^it\ (\text{mod }1)$ lies in an arc $\gamma\subset \mathbb{T}$ of measure $1/b$, and 
\begin{align*}
    \frac{1}{\|-v_{i+1}+b^it\|_\mathbb{T}}\leq \sup_{0\leq s\leq 1/b}\frac{1}{\|-v_{i+1}+t_{i+1}/b+s\|_\mathbb{T}}:=f(t_{i+1},v_{i+1}).
\end{align*} Now, as $\ell$ ranges over $0,...,b^{n}-1$, the numbers $(t+\ell/b^n)$ are $1/b^n$-separated on $\mathbb{T}$, and so the vector of digits $(t_1,...,t_n)$ are distinct for each translation. This provides that \begin{align*}
    \sup_t\sum_{0\leq \ell<b^n}\prod_{i=0}^{n-1}\min\Big\{2d+1,\frac{1}{2\|-v_{i+1}+b^i(t+\ell/b^n)\|_\mathbb{T}}\Big\}\\ \leq \sup_t \sum_{0\leq \ell<b^n}\prod_{i=0}^{n-1}\min\Big\{2d+1,f((t+\ell/b^n)_{i+1},v_{i+1})/2\Big\} \\ =\sum_{0\leq t_1,...,t_n<b}\prod_{i=0}^{n-1}\min \Big\{2d+1,f(t_{i+1},v_{i+1})/2\Big\}.
\end{align*} Thus, \begin{align*}
    \|A\|\leq b^{-n}\int_{\mathbb{T}^n}|\widehat{\mathbf{1}_W}(v)|\sum_{0\leq t_1,...,t_n<b}\prod_{i=0}^{n-1}\min\{2d+1,f(t_{i+1},v_{i+1})/2\}dv.
\end{align*} Since \begin{align*}
    \sum_{0\leq t_1,...,t_n<b}\prod_{i=0}^{n-1}\min\{2d+1,f(t_{i+1},v_{i+1})/2\}=\prod_{i=0}^{n-1}\sum_{0\leq t<b}\min\{2d+1,f(t,v_{i+1})/2\},
\end{align*} we then have the bound \begin{align*}
    \|A\|\leq b^{-n}\int_{\mathbb{T}^n}|\widehat{\mathbf{1}_W}(v)|\prod_{i=0}^{n-1}\sum_{0\leq t<b}\min\{2d+1,f(t,v_{i+1})/2\}dv.
\end{align*} The key point is that this sum may be explicitly evaluated; we may compute \begin{align*}
    \sum_{0\leq t<b}\min\{2d+1,f(t,v_{i+1})/2\}\leq 2(2d+1)+2b\log b
\end{align*} uniformly over $v_{i+1}$; along with the assumption that $2d<b$ this provides that 
\[
    \|A\|\leq (2+2\log b)^n\int_{\mathbb{T}^n}|\widehat{\mathbf{1}_W}(v)|dv.\qedhere
\]
\end{proof}

One may also get two-sided control in the case that the basis vectors are sufficiently lacunary.

\begin{lemma}
    Suppose two finite sets $S\subset \mathbb{Z}^m$ and $S'\subset \mathbb{Z}^n$ are Freiman $H$-isomorphic for some integer $H\geq 1$, with isomorphism $\psi:S\rightarrow S'$. Then, for trigonometric polynomials 
    \begin{align*}
        f(x):=\sum_{s\in S}a_se(s\cdot x),\quad g(x):=\sum_{s\in S}a_{s}e(\psi(s)\cdot x)
    \end{align*} with complex coefficients satisfying $\|f\|_{L^\infty(\mathbb{T}^m)},\|g\|_{L^\infty(\mathbb{T}^n)}\leq K$, one has that \begin{align*}
        \Big|\|f\|_{L^1(\mathbb{T}^m)}-\|g\|_{L^1(\mathbb{T}^n)}\Big|\ll \frac{K}{H}.
    \end{align*}
\end{lemma}

\begin{remark} If $\psi$ is a linear map, so that $\psi(x)=Mx$ for some $n\times m$ matrix, then $\|g\|_\infty\leq \|f\|_\infty$, and so the inequality becomes \begin{align*}
    |\|f\|_1-\|g\|_1|\ll \frac{\|f\|_\infty}{H}.
\end{align*}
\end{remark}

\begin{proof}
    First, note that by definition of $\psi$, \begin{align*}
        \|f\|_{L^{2c}(\mathbb{T}^m)}=\|g\|_{L^{2c}(\mathbb{T}^n)}\quad \text{for every integer }1\leq c\leq H.
    \end{align*}

    Consider the absolute value function on $[-1,1]$. We expand it into a series involving Chebyshev polynomials. First, note that the Fourier series expansion of $|\cos \theta|$ is \begin{align*}
        |\cos \theta|=\frac{2}{\pi}+\frac{4}{\pi}\sum_{k\geq 1}\frac{(-1)^{k+1}}{4k^2-1}\cos(2k\theta)
    \end{align*} and so, if $T_j$ is the $j$th Chebyshev polynomial, one has that \begin{align*}
        |\cos \theta|=\frac{2}{\pi}+\frac{4}{\pi}\sum_{k\geq 1}\frac{(-1)^{k+1}}{4k^2-1}T_{2k}(\cos \theta),
    \end{align*} since by definition $T_{j}(\cos \theta)=\cos(j\theta)$. Then, since cosine is surjective on $[-1,1]$, we deduce that \begin{align}
        |x|=\frac{2}{\pi}+\frac{4}{\pi}\sum_{k\geq 1}\frac{(-1)^{k+1}}{4k^2-1}T_{2k}(x),\quad |x|\leq 1.
    \end{align}

Let \begin{align*}
    Q_H(x):=\frac{2}{\pi}+\frac{4}{\pi}\sum_{1\leq k\leq H}\frac{(-1)^{k+1}}{4k^2-1}T_{2k}(x)
\end{align*} be the truncation of this series; one observes that \begin{align*}
    \Big||x|-Q_H(x)\Big|\ll \sum_{k> H}\frac{|T_{2k}(x)|}{4k^2-1}\ll \frac{1}{H}\quad (|x|\leq 1),
\end{align*} since $|T_{2k}|\leq 1$ for $|x|\leq 1$. The function $Q_H$ is an even polynomial, and thus may be written $Q_H(x)=R_H(x^2)$ for a real polynomial of degree at most $H$. We then have that \begin{align}\label{eq:sqRootApprox}
    \Big|\sqrt{x}-R_H(x)\Big|\ll \frac{1}{H}\quad (0\leq x \leq 1).
\end{align} Using (\ref{eq:sqRootApprox}) and scaling $f$ and $g$ to be 1-bounded, we produce that \begin{align*}
    \Big|\|f\|_{L^1(\mathbb{T}^m)}-K\int_{\mathbb{T}^m}R_H(|f(x)|^2/K^2)dx\Big|\ll \frac{K}{H},\\\Big|\|g\|_{L^1(\mathbb{T}^n)}-K\int_{\mathbb{T}^n}R_H(|g(x)|^2/K^2)dx\Big|\ll \frac{K}{H}.
\end{align*} On the other hand, writing $R_H(y)=\sum_{c=0}^{H}a_c y^c$, we may compute \begin{align*}
    \int_{\mathbb{T}^m}R_H(|f(x)|^2/K^2)dx&=\int_{\mathbb{T}^m}\sum_{c=0}^H a_c|f(x)|^{2c}/K^{2c}dx \\ &= \sum_{c=0}^Ha_cK^{-2c}\int_{\mathbb{T}^m}|f(x)|^{2c}dx \\ &=\sum_{c=0}^{H}a_cK^{-2c} \int_{\mathbb{T}^n}|g(x)|^{2c}dx \\ &= \int_{\mathbb{T}^n}R_H(|g(x)|^2/K^2)dx.
\end{align*} Consequently by the triangle inequality, we produce the result.
\end{proof}

We may then get the following result.

\begin{theorem}\label{thm:linear_iso}
    Suppose $n\geq 2$ is a positive integer. Suppose $S\subseteq [N]^n$, $\psi:S\rightarrow \mathbb{Z}$, $\psi(x_1,\cdots,x_n)=\sum_{i=1}^nx_id_i$, where $d_1,\cdots d_n$ are positive integers such that $d_{i+1}>KN(d_1+\cdots +d_i)$ for some integer $K\geq 1$ and all $1\leq i<n$. Then, for any trigonometric polynomial $g(x)=\sum_{s\in S}a_se(s\cdot x)$ with $\|g\|_\infty\leq K$, we have \begin{align*}
\int_{\mathbb{T}^n}|g(x)|dx=\int_\mathbb{T}\left|\sum_{s\in S}a_se(\psi(s)t)\right|dt+O(1).
    \end{align*} Here, the constant term $O(1)$ does not depend on $N$, $n$, or $K$.
\end{theorem}

\begin{proof}
We claim that $\psi$ is a Freiman isomorphism of order $K$. The forward direction is obvious. Now suppose that \begin{align*}
    \psi(\mathbf{x}_1)+\cdots +\psi(\mathbf{x}_K)=\psi(\mathbf{x}_{K+1})+\cdots +\psi(\mathbf{x}_{2K}).
\end{align*} Then, \begin{align*}
    \sum_{i=1}^nd_i(x_{1,i}+\cdots +x_{K,i})=\sum_{i=1}^nd_i(x_{K+1,i}+\cdots +x_{2K,i}),
\end{align*} and so \begin{align*}
    \sum_{i=1}^{n-1}d_i(x_{1,i}+\cdots +x_{K,i}-x_{K+1,i}-\cdots -x_{2K,i})=d_n(x_{K+1,n}+\cdots +x_{2K,n}-x_{1,n}-\cdots -x_{K,n}).
\end{align*} Suppose that the right-hand side is nonzero; then, by taking absolute values and using the triangle inequality, \begin{align*}
    d_n\leq  \sum_{i=1}^{n-1}d_i\cdot K(N-1)=K(N-1)(d_1+\cdots +d_{n-1}).
\end{align*} This is a contradiction by the growth condition, and so \begin{align*}
    x_{K+1,n}+\cdots +x_{2K,n}=x_{1,n}+\cdots +x_{K,n}.
\end{align*} Iterating this argument gives the desired claim.\\

Using the previous lemma then provides that if \begin{align*}
    f(t):=\sum_{s\in S}a_se(\psi(s)t),
\end{align*} 
then $\|f\|_{L^\infty(\mathbb T)}\leq \|g\|_{L^\infty(\mathbb T^n)}\leq K$, and 
\[
    \Big|\|f\|_{L^1(\mathbb{T})}-\|g\|_{L^1(\mathbb{T}^n)}\Big|=O(1).\qedhere
\]
\end{proof}

\begin{remark}
    If we take $g(x) = \sum_{s\in S}e(s\cdot x)$ and $K = |S|$, then Theorem~\ref{thm:linear_iso} implies that $\|S\| = \|\psi(S)\| + O(1)$. In particular, we can construct subsets of $\mathbb Z$ whose Fourier $L^1$ norm exhibits different types of growth by choosing a suitable multidimensional set $S$. Estimates for multidimensional subsets of $\mathbb Z$ were also studied in \cite{Hanson_2021,Giorgis_2013}, where their results apply to a broader class of sets but only give lower bounds on the Fourier $L^1$ norm.
\end{remark}

\section{Midpoints, convexity, and three-term progressions}\label{sec:3APs}

We now discuss the proof of Theorem \ref{thm:3-AP}, conditional on Theorem \ref{thm:decouplingFourierNorm}, which will be shown later.

\begin{proof}[Proof of Theorem \ref{thm:3-AP} assuming Theorem \ref{thm:decouplingFourierNorm}]
    Suppose that $A\subset \mathbb{Z}^d$, $|A|\geq 2$, satisfies $\|A\|\leq \exp(\log^c|A|)$ for a constant $0<c<1/6$. We note that the trivial partition of $A'\subset A$ satisfies the $L^4$ decoupling inequality (\ref{eq:intro:decoupling}) with constant $K=1$. Applying Theorem \ref{thm:decouplingFourierNorm} then gives that \begin{align*}
        \text{Cap}_2(A')\leq |A| \|A\|^2.
    \end{align*} By considering the uniform probability measure on $A'$, we observe that \begin{align*}
        \text{Cap}_2(A')\geq \frac{|A'|^4}{E_2(A')}
    \end{align*} and so \begin{align*}
        E_2(A')\geq \frac{|A'|^4}{|A| \|A\|^2}.
    \end{align*} Since $|A'|\geq |A|/\|A\|$, we then deduce that $E_2(A')\geq \frac{|A'|^3}{\|A\|^3}$. Assume for contradiction that $A'$ is 3-AP free. By applying the polynomial Balog--Szemerédi--Gowers theorem (see for example~\cite{TaoVu_2009}) to the set $A'$, one is supplied a set $B\subset A'$ and an absolute constant $C_0$ such that \begin{align*}
        |B|\geq \|A\|^{-C_0}|A'|,\quad |B+B|\leq \|A\|^{C_0}|B|. 
    \end{align*} By Pl{\" u}nnecke's inequality, $|2B-2B|\leq \|A\|^{C_1}|B|$ for another absolute constant $C_1$. Ruzsa's modeling lemma~\cite{Ruzsa_1992} then supplies $B_1\subset B$, with $|B_1|\geq |B|/2$, and a Freiman 2-isomorphism from $B_1$ into $\mathbb{Z}/q\mathbb{Z}$, where $q\leq 2\|A\|^{C_1}|B|$. The image $B'\subset \mathbb Z/q\mathbb Z$ is progression-free and has density $|B'|/q\geq \|A\|^{-C_2}$ for an absolute constant $C_2$.

    Now, partition $\mathbb{Z}/q\mathbb{Z}$ into  three intervals $I_1,I_2,I_3$ of length $\lfloor q/3\rfloor$ or $q-2\lfloor q/3\rfloor$. There exists some $i\in [3]$ such that $|B'\cap I_i|\gg q\|A\|^{-C_2}$. Identifying this interval with an interval of integers introduces no wraparound three-term progression. Using the quantitative version of Roth's theorem following the work of Kelley--Meka~\cite{Kelley-Meka_2023}, Bloom--Sisask~\cite{Bloom-Sisask_2023} and Raghavan~\cite{raghavan_2026}, we obtain \begin{align}\label{eq:kelley-meka}
        \|A\|^{-C_2}\ll \exp(-c_0(\log q)^{1/6}/\log\log q)
    \end{align} for an absolute constant $c_0$. On the other hand, $|A'|\geq |A|/\|A\|$, and we have that $q\geq |B'|=|B_1|\gg |B|\gg |A| \|A\|^{-C_4}$. By assumption $\log \|A\|\leq \log^c|A|$, and so for sufficiently large $|A|$, $\log q\geq \frac{1}{2}\log |A|$. This makes (\ref{eq:kelley-meka}) impossible, which provides the desired contradiction.
\end{proof}

\begin{remark} The principal Fourier-analytic input in the preceding proof is not new. Indeed, Corollary 5.5 of \cite{bedert2025largesumfreesubsetssets}, applied with $f=\mathbf{1}_A$, gives for every $B\subseteq A$ the hereditary energy estimate $E_2(B)\gg \frac{|B|^4}{|A|\|A\|^2}$. The rest of the proof follows the same route. An earlier closely related precursor appears in Example 48 of Shkredov \cite{shkredov}, who proves a hereditary mixed-energy lower bound for subsets of a set with small Wiener norm. 
\end{remark}

Our additive capacity argument recovers the displayed estimate from the one-part partition. The point of the theorem here is not a new hereditary-energy inequality, but its geometric interpretation, and its contrast with the finer partition-sensitive conclusions proved for spherical models.\\

A well-known result of Darmon and Merel~\cite{Darmon-Merel_1997} implies that $T_k = \{n^k:n\in \mathbb{N}\}$ is 3-AP free for any $k\geq 3$. As a consequence, we produce the following lower-bound estimate for variants of Weyl sums, which may be of independent interest.

\begin{corollary}
    Let $k\geq 3$ be fixed, and let $A$ be a finite subset of $\mathbb{N}$ with $|A|$ sufficiently large. Then, there exists an absolute constant $c>0$ such that \begin{align*}
        \int_0^1\Big|\sum_{n\in A}e(\alpha n^k)\Big|d\alpha >\exp(\log^{c}|A|).
    \end{align*}
\end{corollary}

An interesting open question, probing the sharpness of the threshold in Theorem \ref{thm:3-AP}, is the following: given $\epsilon>0$, can one construct a large 3-AP-free set $A\subset \mathbb{Z}$ such that $\|A\|\leq |A|^\epsilon$? A non-exhaustive computer search gives a 3-AP-free set \begin{align*}
    A_N=\Big\{\sum_{i=0}^{N-1} d_i4^i:d_i\in \{0,1\}\Big\}
\end{align*} such that $\|A_N\|\leq |A_N|^\epsilon$ with $\epsilon< 0.3336$ for $N=9$, say. This beats the trivial example of the two point set $K=\{n,n+1\}$, which satisfies $\|K\|=|K|^{\delta}$ with $\delta = \frac{\log(4/\pi )}{\log 2}> 0.3485$.  It remains open whether arbitrarily small positive exponents are possible.

\section{Additive capacity and Freiman transference}\label{sec:additiveCapacity}

To produce Theorem \ref{thm:iteratedDimensionReduction} and the more general theorem following, we require a notion of what we call \textit{additive capacity}. Let $\mu$ be a finitely supported probability measure on $\mathbb{Z}^d$, and define the associated weighted $h$-fold additive energy by \begin{align*}
    \mathcal{E}_h(\mu):=\sum_{z\in \mathbb{Z}^d}(\mu^{*h}(z))^2=\sum_{x_1+\cdots+x_h=y_1+\cdot +y_h}\prod_{i=1}^h\mu(x_i)\mu(y_i)=\|\widehat{\mu}\|_{L^{2h}(\mathbb{T}^d)}^{2h}.
\end{align*} If $X_1,\cdots, X_h$, $Y_1,\cdots,Y_h$ are independent random variables with law $\mu$, then \begin{align*}
    \mathcal{E}_h(\mu)=\mathbb{P}(X_1+\cdots +X_h=Y_1+\cdots+Y_h),
\end{align*} and so this weighted additive energy has a natural probabilistic interpretation.

\begin{definition}[Additive capacity] Let $G$ be an abelian group. For a nonempty finite set $S\subset G$ and an integer $h\geq 2$, define \begin{align*}
    \text{Cap}_h(S):=\Big(\inf_{\mu \in \mathcal{P}(S)}\mathcal{E}_h(\mu)\Big)^{-1},
\end{align*} where $\mathcal{P}(S)$ denotes the set of probability measures supported on $S$ (a standard compactness argument shows that this infimum exists). We also define $\mathrm{Cap}_h(\emptyset) = 0$.
\end{definition}

We observe that $\text{Cap}_h(S)$ is large when one can distribute mass on $S$ in such a way that the two random $h$-fold sums rarely collide. Also, this quantity is increasing with respect to inclusion: \begin{align*}
    D\subset A\implies \mathrm{Cap}_h(D)\leq \mathrm{Cap}_h(A).
\end{align*}

\begin{lemma}[Capacity is hereditary]\label{lem:hedCapacity}
Let $A\subset \mathbb{Z}^d$ be finite and nonempty. For every integer $h\geq2 $, every $D\subset A$, and every probability measure $\mu\in \mathcal{P}(D)$, one has that \begin{align*}
    \mathcal{E}_h(\mu)\geq \frac{1}{|A|\|A\|^{2h-2}}.
\end{align*} Equivalently, $\mathrm{Cap}_h(D)\leq \|A\|^{2h-2}|A|$.
\end{lemma}

\begin{proof}
    Put $p=\frac{2h}{2h-1}$. Since $\mu$ is supported on $D\subset A$ and has total mass one, Parseval's identity provides that \begin{align*}
        1=\sum_{x\in \mathbb{Z}^d}\mathbf{1}_A(x)\mu(x)=\int_{\mathbb{T}^d}\widehat{A}(\xi)\overline{\widehat{\mu}(\xi)}d\xi.
    \end{align*} By H\"older's inequality and then interpolation of $L^p$ between $L^1$ and $L^2$, \begin{align*}
        1\leq \|\widehat{\mathbf{1}_A}\|_p\|\widehat{\mu}\|_{2h}\leq \|A\|^{1-1/h}\|\widehat{\mathbf{1}_A}\|_2^{1/h}\|\widehat{\mu}\|_{2h}.
    \end{align*} Since $\|\widehat{\mathbf{1}_A}\|_2^2=|A|$, after raising to the $2h$-th power then deduce that \begin{align*}1\leq \|A\|^{2h-2}|A|\cdot \|\widehat{\mu}\|_{2h}^{2h}.\end{align*} Then, using that $\|\widehat{\mu}\|_{2h}^{2h}=\mathcal{E}_h(\mu)$ and taking infimum over $\mu\in \mathcal{P}(D)$, we have that \begin{align*}
        1\leq \|A\|^{2h-2}|A|\inf_{\mu\in \mathcal{P}(D)}\mathcal{E}_h(\mu),
    \end{align*} and so 
    \[
        \text{Cap}_h(D)=\Big(\inf_{\mu\in \mathcal{P}(D)}\mathcal{E}_h(\mu)\Big)^{-1}\leq \|A\|^{2h-2}|A|.\qedhere
    \]
\end{proof}

\textbf{Remark}. For the uniform measure $\mu=|D|^{-1}\mathbf{1}_D$, Lemma \ref{lem:hedCapacity} becomes $E_h(D)\geq \frac{|D|^{2h}}{\|A\|^{2h-2}|A|}$, where $E_h(D)$ is the higher additive energy; this recovers Bedert~\cite[Corollary 5.5]{bedert2025largesumfreesubsetssets} for $f=\mathbf{1}_A$ and $h=2$. On the other hand, it is also not difficult to see that $\mathrm{Cap}_h(D)\geq |D|$.

\begin{lemma}[Capacity transfers through Freiman models]\label{lem:capacityTransfer} Let $A\subset \mathbb{Z}^m$, $D\subset A$, and let $\phi:D\rightarrow D'\subset \mathbb{Z}^d$ be a bijective Freiman $h$-homomorphism. Then, \begin{align*}
    \mathrm{Cap}_h(D')\leq \|A\|^{2h-2}|A|.
\end{align*}
\end{lemma}

\begin{proof}
    Fix $\nu\in \mathcal{P}(D')$ and pull back to $D$ by setting $\mu(d)=\nu(\phi(d))$. Because $\phi$ is bijective, this gives a bijection between $\mathcal{P}(D')$ and $\mathcal{P}(D)$. Every additive relation counted by $\mathcal{E}_h(\mu)$ maps to a corresponding one counted by $\mathcal{E}_h(\nu)$, and the corresponding weights are equal, so $\mathcal{E}_h(\mu)\leq \mathcal{E}_h(\nu)$. Lemma \ref{lem:hedCapacity} gives that $\mathcal{E}_h(\nu)\geq \mathcal{E}_h(\mu)\geq \frac{1}{\|A\|^{2h-2}|A|}$, and so taking the infimum over $\nu\in \mathcal{P}(D')$ provides the result.
\end{proof}

\begin{remark} If $\phi$ is a Freiman $h$-isomorphism, then the two weighted energies are equal (and thus the capacities are as well).
\end{remark}

\section{Super-additivity of additive capacity under partition decoupling}\label{sec:partitionDecoupling}

One key property of this additive capacity notion is that it has a sort of super-additivity that is useful in applications. First, let us recall some terminology from decoupling theory. We specialize to the integer lattice $\mathbb{Z}^d$ and its dual $\mathbb{T}^d$ for concreteness.

\begin{definition}[Partition-decoupling] Let $S\subset \mathbb Z^d$ be a finite set and $S=\bigsqcup_{\theta\in \Theta}S_\theta$ be a partition; take an integer $h\geq 2$. We say that the partition has $2h$-decoupling constant $K$ if \begin{align}\label{eq:partitionDecoupling}
    \Big\|\sum_{s\in S}a(s)e(s\cdot )\Big\|_{L^{2h}(\mathbb{T}^d)}\leq K\Big(\sum_{\theta\in \Theta}\Big\|\sum_{s\in S_\theta}a(s)e(s\cdot )\Big\|_{L^{2h}(\mathbb{T}^d)}^2\Big)^{1/2}
\end{align} for every $a: S\rightarrow \mathbb{C}$.
\end{definition}

\begin{remark} In practice, we will choose $S$ to be a subset of a larger ambient object (e.g. a high-dimensional sphere), and so the sets $S_\theta$, which are inherited from a partition of this larger object, may well be empty. \end{remark}

The consequence is then as follows.

\begin{theorem}[Super-additivity of additive capacity]\label{thm:ACsubAdditive}
    Under the partition-decoupling hypothesis (\ref{eq:partitionDecoupling}), one has that \begin{align*}
        \mathrm{Cap}_h(S)^{1/h}\geq K^{-2}\sum_{\theta\in \Theta}\mathrm{Cap}_h(S_\theta)^{1/h}.
    \end{align*} Here, the sum runs over nonempty $S_\theta$.
\end{theorem}

\begin{proof}
    Take $\eta>0$. Relabel $(S_\theta)_{\theta\in \Theta}$ as necessary to exclude empty sets. For each remaining $S_\theta$, choose $\mu_\theta\in \mathcal{P}(S_\theta)$ such that \begin{align}\label{eq:partitionNearMinimizer}
        \Big\|\sum_{s\in S_\theta}\mu_\theta(s)e(s\cdot )\Big\|_{2h}\leq (1+\eta)\text{Cap}_h(S_\theta)^{-1/2h}
    \end{align} (this is essentially choosing a measure $\mu_\theta$ that is a near-minimizer to the one that attains the infimum in the definition of $\text{Cap}_h(S_\theta)$). Let $(t_\theta)_{\theta\in \Theta}\geq 0$ with $\sum_{\theta\in \Theta}t_\theta=1$, and set $\mu=\sum_\theta t_\theta \mu_\theta$, which is a probability measure on $S$.  Applying (\ref{eq:partitionDecoupling}) and (\ref{eq:partitionNearMinimizer}), we obtain that \begin{align*}
        \Big\|\sum_{s\in S}\mu(s)e(s\cdot)\Big\|_{2h}\leq K(1+\eta)\Big(\sum_{\theta\in\Theta}t_\theta^2\text{Cap}_h(S_\theta)^{-1/h}\Big)^{1/2}.
    \end{align*} The expression on the right is minimized, subject to $\sum_\theta t_\theta =1$, by \begin{align*}
        t_\theta= \frac{\text{Cap}_h(S_\theta)^{1/h}}{\sum_{\theta'}\text{Cap}_h(S_{\theta'})^{1/h}}.
    \end{align*} Using this choice of coefficients then provides that $\sum_{\theta}t_\theta^2\text{Cap}_h(S_\theta)^{-1/h}=(\sum_\theta \text{Cap}_h(S_\theta)^{1/h})^{-1}$, and so \begin{align*}
        \inf_{\mu\in \mathcal{P}(S)}\Big\|\sum_{s\in S}\mu(s)e(s\cdot )\Big\|_{2h}\leq K(1+\eta)\Big(\sum_\theta \text{Cap}_h(S_\theta)^{1/h}\Big)^{-1/2}.
    \end{align*} Raising to the $2h$-th power provides then that \begin{align*}
       \inf_{\mu\in \mathcal{P}(S)}\mathcal{E}_h(\mu)\leq K^{2h}(1+\eta)^{2h}\Big(\sum_\theta\text{Cap}_h(S_\theta)^{1/h}\Big)^{-h},
    \end{align*} and so \begin{align*}
        \text{Cap}_h(S)\geq K^{-2h}(1+\eta)^{-2h}\Big(\sum_\theta\text{Cap}_h(S_\theta)^{1/h}\Big)^h.
    \end{align*} Since $\eta>0$ was arbitrary, taking $\eta\rightarrow 0$ then provides that \begin{align*}
        \text{Cap}_h(S)\geq K^{-2h}\Big(\sum_\theta \text{Cap}_h(S_\theta)^{1/h}\Big)^h,
    \end{align*} and taking $h$-th roots then provides the result.
\end{proof}

The following is a more general formulation of Theorem \ref{thm:masterThm}.

\begin{theorem}[Partition-decoupling forces large Fourier norm]\label{thm:decouplingFourierNorm}
    Let $A\subset \mathbb{Z}^m$ be finite, and suppose that $D\subset A$ admits a bijective Freiman $h$-homomorphism onto $S\subset \mathbb{Z}^n$. Suppose $S=\bigsqcup_{\theta\in \Theta}S_\theta$ satisfies the partition-decoupling condition (\ref{eq:partitionDecoupling}) with constant $K$. Then, \begin{align*}
        \sum_{\theta\in \Theta}\mathrm{Cap}_h(S_\theta)^{1/h}\leq K^2|A|^{1/h}\|A\|^{2-2/h}.
    \end{align*}
\end{theorem}

\begin{proof}
     By Theorem \ref{thm:ACsubAdditive} one obtains that \begin{align*}
        \text{Cap}_h(S)^{1/h}\geq K^{-2}\sum_\theta\mathrm{Cap}_h(S_\theta)^{1/h}.
    \end{align*} Lemma \ref{lem:capacityTransfer} gives that $\text{Cap}_h(S)^{1/h}\leq |A|^{1/h}\|A\|^{2-2/h}$. Combining these two inequalities then gives that 
    \[
        \sum_{\theta\in\Theta}\text{Cap}_h(S_\theta)^{1/h}\leq K^2|A|^{1/h}\|A\|^{2-2/h}. \qedhere
    \]
\end{proof}

\begin{corollary}[Concentration bound]\label{cor:concentration} Under the hypotheses of Theorem \ref{thm:decouplingFourierNorm}, let \begin{align*}
    T:=\#\{\theta\in\Theta : S_\theta\neq \varnothing\},\quad n_{max}(S,\Theta):=\max_{\theta\in \Theta}|S_\theta|.
\end{align*} Then, one has \begin{enumerate}
    \item[(1)] An upper bound on $T$: \begin{align*}
        T \leq K^2|A|^{1/h}\|A\|^{2-2/h}.
    \end{align*}
    \item[(2)] A lower bound on $n_{max}(S,\Theta)$: \begin{align*}
        n_{max}(S,\Theta)\geq \Big(\frac{|S|}{K^2|A|^{1/h}\|A\|^{2-2/h}}\Big)^{\frac{h}{h-1}}.
    \end{align*}
\end{enumerate} In addition, if $|S_\theta|\leq U$ for every $\theta$ (so that $S$ is well-distributed w.r.t. the partition), then \begin{align*}
     \|A\|\geq \Big(\frac{|S|}{U^{1-1/h}K^2|A|^{1/h}}\Big)^{\frac{h}{2h-2}}.
\end{align*}
    
\end{corollary}

\begin{proof}
    First, note that from $\S\ref{sec:additiveCapacity}$, we have that $\text{Cap}_h(S_\theta)\geq |S_\theta|$. Applying this to Theorem \ref{thm:decouplingFourierNorm} then gives that \begin{align*}
        \sum_{\theta}|S_\theta|^{1/h}\leq K^2|A|^{1/h}\|A\|^{2-2/h}.
    \end{align*} Every nonempty $S_\theta$ contributes at least one to the left-hand side of the statement in Theorem \ref{thm:decouplingFourierNorm}, and so $T\leq K^2|A|^{1/h}\|A\|^{2-2/h}$, providing (1). On the other hand, since $|S_\theta|\leq n_{max}(S,\Theta)$ for each $\theta\in \Theta$ and $1/h-1<0$, \begin{align*}
        \sum_{\theta}|S_\theta|^{1/h}=\sum_{\theta:S_\theta\neq\emptyset} |S_\theta|\cdot |S_\theta|^{1/h-1}\geq n_{max}(S,\Theta)^{1/h-1}\sum_\theta|S_\theta|=n_{max}(S,\Theta)^{1/h-1}|S|.
    \end{align*} Combining this with Theorem \ref{thm:decouplingFourierNorm} then provides that \begin{align*}
        n_{max}(S,\Theta)^{1/h-1}|S|\leq K^2|A|^{1/h}\|A\|^{2-2/h}.
    \end{align*} Rearranging then provides that \begin{align*}
        n_{max}(S,\Theta)\geq \Big(\frac{|S|}{K^2|A|^{1/h}\|A\|^{2-2/h}}\Big)^{\frac{h}{h-1}},
    \end{align*} yielding (2). Finally, if $|S_\theta|\leq U$ for all $\theta$, then \begin{align*}
        |S|=\sum_\theta |S_\theta|\leq U^{1-1/h}\sum_\theta |S_\theta|^{1/h}\leq U^{1-1/h}K^2|A|^{1/h}\|A\|^{2-2/h},
    \end{align*} and so 
    \[
        \|A\|\geq \Big(\frac{|S|}{U^{1-1/h}K^2|A|^{1/h}}\Big)^{\frac{h}{2h-2}}.\qedhere
    \]
\end{proof}

Let us next define hereditary energy. For $h=2$, the following definition is due to Sanders \cite{sandersEnergy}; the higher-order version here is the more general $h$-fold analog.

\begin{definition}[Hereditary $h$-energy] Let $0<\nu\leq 1$. A nonempty finite set $X\subset G$ is called \emph{$\nu$-hereditarily $h$-energetic}, if, whenever $Y\subseteq X$ and $|Y|\geq \delta |X|$ with $0<\delta\leq 1$, one has \begin{align*}
    E_h(Y)\geq \nu\delta |Y|^{2h-1}.
\end{align*} For $h=2$ we simply say $\nu$-hereditarily energetic.
    
\end{definition}

\begin{remark} Large energy alone does not imply hereditary energy: Sanders \cite[Section 3]{sandersEnergy} gives the example of adjoining a large dissociated set with a high-energy set. Hereditarily energetic sets are significantly more structured than those with simply high energy (on the spectrum of structure, they lie between small-doubling and high energy), since they require large energy on every dense subset.\end{remark}

\begin{theorem}[Hereditary energy on good parts]\label{thm:herEnergy} Assume the hypotheses of Theorem \ref{thm:decouplingFourierNorm} and $D\neq \emptyset$. Let $\alpha:=|D|/|A|=|S|/|A|$. Then, for every $\rho\in (0,1)$, there is a set of indices $\Theta_{good}\subseteq \Theta$ such that $|S_\theta|>0$ for $\theta\in \Theta_{good}$, \begin{align*}
    |\Theta_{good}|\leq K^2|A|^{1/h}\|A\|^{2-2/h},\quad \sum_{\theta\in \Theta_{good}}|S_\theta|\geq (1-\rho)|S|,
\end{align*} and every $S_\theta$ with $\theta\in \Theta_{good}$ is $\nu$-hereditarily $h$-energetic, where \begin{align*}
    \nu=\frac{\rho \alpha}{K^{2h}\|A\|^{2h-2}}.
\end{align*} Moreover, for each $\theta\in \Theta_{good}$, \begin{align*}
    \mathrm{Cap}_h(S_\theta)\leq \frac{K^{2h}\|A\|^{2h-2}}{\rho \alpha}|S_\theta|.
\end{align*}
    
\end{theorem}

\begin{proof}
    For ease of notation, set $Q:=\frac{K^{2h}\|A\|^{2h-2}}{\rho \alpha}$. For $\theta\in \Theta$, call $\theta$ bad if $|S_\theta|>0$ and $\text{Cap}_h(S_\theta)>Q|S_\theta|$. If $M_{bad}:=\sum_{\theta\text{ bad}}|S_\theta|$, then Theorem \ref{thm:decouplingFourierNorm} and the inequality $\sum_j x_j^{1/h}\geq (\sum_j x_j)^{1/h}$ give \begin{align*}
        K^2|A|^{1/h}\|A\|^{2-2/h}&\geq \sum_{\theta\text{ bad}}\text{Cap}_h(S_\theta)^{1/h} \\ &>Q^{1/h}\sum_{\theta\text{ bad}}|S_\theta|^{1/h} \\ &\geq Q^{1/h}M_{bad}^{1/h}.
    \end{align*} Raising to the $h$th power and substituting $|A|=|S|\alpha^{-1}$ gives that \begin{align*}
        M_{bad}\leq Q^{-1}K^{2h}|S|\alpha^{-1}\|A\|^{2h-2}=\rho |S|.
    \end{align*} Thus, $\sum_{\theta\not\in \Theta_{bad}}|S_\theta|\geq (1-\rho)|S|.$ We will set $\Theta_{good}$ to be the set of indices $\theta$ that are not bad, and such that $|S_\theta|>0$.\\

    The size bound on $\Theta_{good}$ simply comes from the concentration bound, Corollary \ref{cor:concentration}.\\

    Now, we show the latter part of the theorem. Fix a good nonempty subset $S_\theta$ and a subset $Y\subseteq S_\theta$ with $|Y|=\delta |S_\theta|$. Note that $\text{Cap}_h(S_\theta)\leq Q|S_\theta|$, and so every probability measure $\mu$ on $S_\theta$ satisfies \begin{align*}
        \mathcal{E}_h(\mu)\geq \frac{1}{Q|S_\theta|}.
    \end{align*} Now, take $\mu$ to be the uniform probability measure on $Y$. Since its weighted energy is $E_h(Y)/|Y|^{2h}$, this provides then that \begin{align*}
        E_h(Y)\geq \frac{|Y|^{2h}}{Q|S_\theta|}=\frac{\delta}{Q}|Y|^{2h-1}.
    \end{align*} Since $Q^{-1}=\nu$, this provides the desired statement.    
\end{proof}

\section{Proofs of the main spherical theorems}\label{sec:mainProofs}

In this section we provide proofs of Theorems \ref{thm:iteratedDimensionReduction}, \ref{thm:mainThm2}, \ref{thm:mainThm3}. 

We begin with some notation. We take a finite set $A\subset \mathbb{Z}^m$ with $D\subseteq A$ Freiman 2-isomorphic to some $S\subseteq \Sigma$ where $\Sigma$ is the center sphere of radius $R$ in $\mathbb{R}^d$. We write $P=c+W$ for an affine subspace, where $W\subset \mathbb{R}^d$ is a subspace containing zero and $c\in W^\perp$. For $z\in \mathbb{R}^d$ we will write $H_z$ for \begin{align*}
    H_z:=\{u\in \mathbb{R}^d:u\cdot z = \frac{|z|^2}{2}\}.
\end{align*}

Let us begin with some general results that will be helpful. We first have a lemma that essentially tells us that large hereditary energy on a sphere allows one to pass to a rich lower-dimensional fiber.

\begin{lemma}\label{lem:richFibers}
    Let $P=c+W\subset \mathbb{R}^d$ have dimension at least 2, and $\varnothing\neq X\subseteq S\cap P$. If $|P\cap \Sigma|\geq 2$ then there exists $z\in X+X$, with $z\neq 2c$, such that \begin{align*}
        |X\cap (z-X)|\geq \frac{|X|^2}{4\mathrm{Cap}_2(X)}.
    \end{align*} Moreover, \begin{align*}
        X\cap (z-X)\subseteq P\cap H_z\cap S,
    \end{align*} where $P\cap H_z$ is a proper affine hyperplane in $P$.
\end{lemma}

\begin{proof}
    We first consider the structure of $P\cap \Sigma$. If $x\in P\cap \Sigma$, then $x=c+w$ for some $w\in W$, and so $R^2=|x|^2=|c+w|^2=|c|^2+|w|^2$; thus, \begin{align*}
        P\cap \Sigma = c+\{w\in W:|w|^2=R^2-|c|^2\}.
    \end{align*} We then claim that we may find $u\in W$ such that $u\cdot (x-c)\neq 0$ for every $x\in X$. Indeed, if $V_x:=\{u\in W:u\cdot (x-c)=0\}$, then for each $x$, $V_x$ is a proper subspace of $W$ (since $x\neq c$, as by assumption $|P\cap \Sigma|\geq 2$), and $W\setminus \bigcup_{x\in X}V_x$ is nonempty by dimensional considerations. Partition $X$ into the two relative open hemispheres determined by the sign of $u\cdot (x-c)$, and let $X_0$ be the larger part. Then, $|X_0|\geq |X|/2$, and no pair $x,y\in X_0$ satisfies $x+y=2c$ (for if $x+y=2c$, then $0=(x-c)+(y-c)$, while $u\cdot (x-c),u\cdot (y-c)$ have the same sign, a contradiction).

    For $z\in X+X$, write $r(z)=|X_0\cap (z-X_0)|$. Since \begin{align*}
        E_2(X_0)=\sum_z r(z)^2,\quad \sum_z r(z)=|X_0|^2,
    \end{align*} provides that \begin{align*}
        \max_z r(z)\geq \frac{E_2(X_0)}{|X_0|^2}\geq \frac{|X_0|^2}{\mathrm{Cap}_2(X_0)}\geq \frac{|X|^2}{4\mathrm{Cap}_2(X)}.
    \end{align*} Here, we used that $E_2(X_0)\geq \frac{|X_0|^4}{\mathrm{Cap}_2(X_0)}$, by considering the uniform probability measure on $X_0$; we also used that $\mathrm{Cap}_2(X_0)\leq \mathrm{Cap}_2(X)$, since $X_0\subseteq X$. Choose a $z$ that maximizes $r(z)$; for such a $z$, we have $z\in X_0+X_0$; hence $z\neq 2c$.

    We now verify the last parts of the lemma. If $x,z-x\in \Sigma$, we have that $|x|^2=|z-x|^2$, and so expanding out the square gives that $x\cdot z=|z|^2/2$. Thus, $X\cap (z-X)\subseteq H_z$. Now consider when $P\subseteq H_z$. We observe that $z\perp W$ in such a case: since $c,c+w\in P$ for all $w\in W$, we have that $(c+w)\cdot z=\frac{|z|^2}{2}$ and $c\cdot z=\frac{|z|^2}{2}$, and subtracting the two gives $w\cdot z=0$ for all $w\in W$. Since $z=x+y$ for some $x,y\in X_0$, and since $x=c+w$, $y=c+v$ for some $w,v\in W$, we force $v+w=0$ by projecting onto $W$, and so $z=2c$, a contradiction. Thus, $P\not\subseteq H_z$, and so $P\cap H_z$ is a proper affine hyperplane in $P$ with codimension 1 in $P$.
\end{proof}

Our next proposition shows that, up to a controlled error, a given partition may be refined into one that is lower-dimensional, depending on the capacity of the parts.

\begin{proposition}\label{prop:capacity controls concentration} Let $R>0$, $d\geq 2$ and let \begin{align*}
    X=\bigsqcup_{i\in I}X_i\subseteq  \mathcal{F}_{d,R^2}
\end{align*} be a partition into nonempty finite sets. Then, for every $\delta\in (0,1)$, there exists a set $X'\subseteq X$, and pairwise disjoint sets $Y_1,\cdots Y_J$ satisfying the following: \begin{enumerate}
    \item[(i)] $|X\setminus X'|\leq \delta |X|$,
    \item[(ii)] $X'= \bigsqcup_{j=1}^JY_j$,
    \item[(iii)] For every $j\in [J]$, there is an index $i(j)\in I$ and a proper affine hyperplane $Q_j\subsetneq \mathbb{R}^d$ such that \begin{align*}
        Y_j\subseteq X_{i(j)}\cap Q_j,
    \end{align*}
    \item[(iv)] The number of lower-dimensional parts is controlled by capacity: \begin{align*}
        J\ll \frac{\Big(\sum_{i\in I}\mathrm{Cap}_2(X_i)^{1/2}\Big)^2}{\delta^2|X|}.
    \end{align*}
\end{enumerate} 
\end{proposition}

\begin{proof}
    Let $B_I:=\sum_{i\in I}\mathrm{Cap}_2(X_i)^{1/2}$. We begin by discarding every index $i\in I$ such that $\frac{\mathrm{Cap}_2(X_i)^{1/2}}{|X_i|}>\frac{2B_I}{\delta|X|}$; call this bad set $I_{bad}$: \begin{align*}
        I_{bad}:=\Big\{i\in I:\frac{\mathrm{Cap}_2(X_i)^{1/2}}{|X_i|}>\frac{2B_I}{\delta |X|}\Big\}.
    \end{align*} Then, \begin{align}\label{eq:capacityProp1}
        \sum_{i\in I_{bad}}|X_i|<\frac{\delta |X|}{2B_I}\sum_{i\in I_{bad}}\mathrm{Cap}_2(X_i)^{1/2}\leq \frac{\delta |X|}{2}.
    \end{align} Let $I_{good}:=I\setminus I_{bad}$. We may then compute  \begin{align*}
        \sum_{i\in I_{good}}\frac{\mathrm{Cap}_2(X_i)}{|X_i|}=\sum_{i\in I_{good}}\mathrm{Cap}_2(X_i)^{1/2}\frac{\mathrm{Cap}_2(X_i)^{1/2}}{|X_i|}\leq \frac{2B_I^2}{\delta |X|}.
    \end{align*} Since $\mathrm{Cap}_2(X_i)\geq |X_i|$, we then deduce that $|I_{good}|\leq \frac{2B_I^2}{\delta |X|}$.

    Now, fix some $X_i$ with $i\in I_{good}$. If $X_i$ is a singleton, then one proper hyperplane can pass through that point, which suffices. Otherwise, we apply Lemma \ref{lem:richFibers} in a greedy manner to $X_i\setminus (H_z\cap \mathcal{F}_{d,R^2})$ for the appropriate $z$ (for the purposes of the application here, we take $P=\mathbb{R}^d$). Let $T:=4\mathrm{Cap}_2(X_i)$. If $m_j$ is the size of the set remaining after $j$ hyperplanes have been found, then \begin{align*}
        m_{j+1}\leq m_j-\frac{m_j^2}{T}.
    \end{align*} Here, $m_j\leq |X_i|\leq T$. If $m_{j+1}>0$, then we have that \begin{align*}
        \frac{1}{m_{j+1}}\geq \frac{1}{m_j}+\frac{1}{T}.
    \end{align*} Iterating this inequality from $m_0=|X_i|$ gives $\frac{1}{m_j}\geq \frac{1}{|X_i|}+\frac{j}{T}$, and so \begin{align*}
        m_j\leq \frac{T}{j+T/|X_i|}.
    \end{align*} Once $j\geq 2T/(\delta |X_i|)$, we see that $m_j\leq \frac{\delta |X_i|}{2}$, and so all but $\frac{\delta |X_i|}{2}$ points of $X_i$ are covered by $O(\frac{T}{\delta |X_i|})$ proper affine subspaces of $\mathbb{R}^d$. Assign each covered point to the any one fixed subspace containing it, so that the children $Y_1,\cdots, Y_J$ are disjoint. The number of new children is at most \begin{align*}
        \sum_{i\in I_{good}}O\Big(\frac{T}{\delta |X_i|}\Big)\ll \frac{1}{\delta}\sum_{i\in I_{good}}\mathrm{Cap}_2(X_i)^{1/2}\cdot \frac{2 B_I}{\delta |X|}\ll \frac{B_I^2}{\delta^2|X|}.
    \end{align*} Setting $X'= \bigsqcup_{j=1}^JY_j$, we may then compute from (\ref{eq:capacityProp1}) and the points left over by the coverings that 
    \[
        |X\setminus X'|\leq \frac{\delta |X|}{2}+\sum_{i}\frac{\delta |X_i|}{2}\leq \delta |X|.\qedhere
    \]
\end{proof}

\subsection{Proof of Theorem \ref{thm:iteratedDimensionReduction}}

We now turn to the proof of Theorem \ref{thm:iteratedDimensionReduction}. This will not require any decoupling other than the trivial example. Rather, we use an induction on dimension here, where we repeatedly apply Lemma \ref{lem:richFibers} down to dimension 3, and then apply a result of Mudgal \cite{mudgal} on additive energies of sets on spheres in this setting, which gives the result.

First, note that from Lemma~\ref{lem:capacityTransfer}, every subset $X\subseteq S$ satisfies \begin{align}\label{eq:capacityControl}
    \mathrm{Cap}_2(X)\leq \mathrm{Cap}_2(S)\leq |A|\|A\|^2.
\end{align}

\begin{proposition}[Capacity of Spherical Sets]\label{prop:capacitySpheres}
    Let $X$ be a finite nonempty subset of a Euclidean sphere in $d$-dimensional affine space, for $d\geq 3$. Then for every $\epsilon>0$, \begin{align*}
        |X|\ll_{d,\epsilon} \mathrm{Cap}_2(X)^{1-\frac{7}{2^{d+1}}+\epsilon}.
    \end{align*}
\end{proposition}

\begin{proof}
    We induct on dimension. First, for dimension $d=3$, one has by considering the uniform measure and applying \cite[Theorem 1.5]{mudgal} that \begin{align*}
        \mathrm{Cap}_2(X)\geq \frac{|X|^4}{E_2(X)}\gg_\epsilon |X|^{\frac{16}{9}-\epsilon}.
    \end{align*} Upon rearranging, this gives us our base case. Now, suppose we have the desired inequality for dimension $d\geq 3$. Let $X$ be a finite nonempty subset of a Euclidean sphere in $(d+1)$-dimensional affine space. By applying Lemma \ref{lem:richFibers}, there exists a proper (i.e. lower-dimensional) affine spherical section $Y\subseteq X$ with \begin{align}\label{eq:prop36eq1}
        |Y|\geq \frac{|X|^2}{4\mathrm{Cap}_2(X)}.
    \end{align} Applying the induction hypothesis and monotonicity of capacity gives that \begin{align}\label{eq:prop36eq2}
        |Y|\ll_{d,\epsilon}\mathrm{Cap}_2(Y)^{1-\frac{7}{2^{d+1}}+\epsilon}\leq \mathrm{Cap}_2(X)^{1-\frac{7}{2^{d+1}}+\epsilon}.
    \end{align} Combining (\ref{eq:prop36eq1}) and (\ref{eq:prop36eq2}) then provides that \begin{align*}
        |X|^2\ll_{d,\epsilon}\mathrm{Cap}_2(X)^{2-\frac{7}{2^{d+1}}+\epsilon}.
    \end{align*} Taking square roots and relabeling $\epsilon$ then closes the induction.
\end{proof}

\begin{proof}[Proof of Theorem \ref{thm:iteratedDimensionReduction}] From (\ref{eq:capacityControl}) and Proposition \ref{prop:capacitySpheres} one has that \begin{align*}
    |S|\ll_{d,\epsilon}(|A|\|A\|^2)^{1-\frac{7}{2^{d+1}}+\epsilon}.
\end{align*} Using $|S|=\alpha |A|$ and rearranging then provides that 
\[
    \|A\|\gg_{d,\epsilon} \alpha^{\frac{1}{2}+\frac{7}{2t_d}-\epsilon}|A|^{\frac{7}{2t_d}-\epsilon},\quad t_d:=2^{d+1}-7.\qedhere
\]
\end{proof}

\subsection{Proof of Theorem \ref{thm:mainThm2}} We now prove Theorem \ref{thm:mainThm2}, where we use the machinery from previous sections. The key input is the following decoupling theorem from Bourgain--Demeter \cite{BourgainDemeter2015}. Recall the notation \begin{align*}
     \mathcal{F}_{d,R^2}=\{x\in \mathbb{Z}^d:|x|^2=R^2\}.
\end{align*} We recall the classical estimate that \begin{align*}
    |\mathcal{F}_{d,R^2}|\asymp_d R^{d-2},\quad d\geq 5.
\end{align*}

By a standard $\sigma$-cap partition of $\mathcal{F}_{d,R^2}$, we mean the partition induced by the usual tangential $\sigma$-scale decomposition of the unit sphere appearing in decoupling theory, such that \begin{align*}
    \mathcal{F}_{d,R^2}=\bigsqcup_{\theta\in \Theta}V_\theta,\quad |\Theta|\asymp \sigma^{-(d-1)},
\end{align*} and every $V_\theta$ is contained in a spherical cap of angular radius $O_d(\sigma)$. We will use the following result as a black box.

\begin{theorem}[Bourgain--Demeter, \cite{BourgainDemeter2015}]\label{thm:blackBox} Let $d\geq 5$, $R^{-1}\leq \sigma\leq c_d$, and let $\{V_\theta\}_{\theta\in\Theta}$ be a standard $\sigma$-cap partition of $\mathcal{F}_{d,R^2}$. If $S\subseteq \mathcal{F}_{d,R^2}$, set $S_\theta:=S\cap V_\theta$. Then, for every $\epsilon>0$ and $a:S\rightarrow \mathbb{C}$, \begin{align*}
    \Big\|\sum_{x\in S}a(x)e(x\cdot \xi)\Big\|_{4(\xi)}\ll_{d,\epsilon} \sigma^{-\frac{d-3}{4}-\epsilon}\Big(\sum_{\theta\in \Theta}\Big\|\sum_{x\in S_\theta}a(x)e(x\cdot \xi)\Big\|_{4(\xi)}^2\Big)^{1/2}.
\end{align*} One also has the global extension estimate \begin{align}\label{eq:globalExtension}
    \Big\|\sum_{x\in S}a(x)e(x\cdot \xi)\Big\|_{4(\xi)}\ll_{d,\epsilon}R^{\frac{d-4}{4}+\epsilon}\|a\|_{\ell^2(S)}.
\end{align}
\end{theorem}

\begin{proof}
    The first assertion is the periodic discrete form of the Bourgain--Demeter $\ell^2$-decoupling theorem at $p=4$, applied at scale $\delta=\sigma^2$; see Appendix A for the standard transference. The global estimate is taken verbatim from their paper.
 \end{proof}

We remark that the estimate (\ref{eq:globalExtension}) can quickly give Theorem \ref{thm:global}.

\begin{proof}[Proof of Theorem \ref{thm:global}]
   Let $\mu=|S|^{-1}\mathbf{1}_S$. Then, $\mu$ is a probability measure on $S$, with \begin{align*}
       \mathcal{E}_2(\mu)^{1/4}=\Big\|\sum_{x\in S}\mu(x)e(x\cdot \xi)\Big\|_{4(\xi)}\ll_\epsilon R^{\frac{d-4}{4}+\epsilon}|S|^{-1/2}
   \end{align*} and so \begin{align*}
       \mathrm{Cap}_2(S)\gg_{d,\epsilon} R^{-(d-4)-\epsilon}|S|^2,
   \end{align*} after relabeling $\epsilon$. Consequently, by Lemma \ref{lem:capacityTransfer}, one has that \begin{align*}
       R^{-(d-4)-\epsilon}|S|^2\ll_{d,\epsilon}\|A\|^{2}|A|.
   \end{align*} Finally, $|S|=\alpha|A|=\beta |\mathcal{F}_{d,R^2}|\asymp_d \beta R^{d-2}$, so $R^{d-2}\asymp_d \alpha |A|/\beta$. Substituting this into the above inequality gives then that 
   \[
       \|A\|\gg_{d,\epsilon} |A|^{\frac{1}{d-2}-\epsilon}\alpha^{\frac{d}{2(d-2)}-\epsilon}\beta^{\frac{d-4}{2(d-2)}+\epsilon}.\qedhere
   \]
\end{proof}

We next require the following result.

\begin{theorem}\label{thm:CapacityDecoupling}
   Suppose $A\subset \mathbb{Z}^m$ has a subset $D\subseteq A$ Freiman 2-isomorphic to a finite set $S\subseteq \mathcal{F}_{d,R^2}$, where $d\geq 5$. Partition $\mathcal{F}_{d,R^2}$ into standard $\sigma$-caps $\bigsqcup_{\theta\in \Theta}V_\theta$, where $R^{-1}\leq \sigma\leq c_d$. Then, if $S_\theta:=S\cap V_\theta$, for sufficiently small $\epsilon>0$ we have\begin{align*}
        \sum_{\theta\in \Theta}\mathrm{Cap}_2(S_\theta)^{1/2}\ll_{d,\epsilon} \sigma^{-(d-3)/2-\epsilon}|A|^{1/2}\|A\|.
    \end{align*} Consequently, \begin{align*}
        \sum_{\theta\in \Theta}|S_\theta|^{\frac{1}{2}+\frac{7}{2t_d}-\epsilon}\ll_{d,\epsilon}\sigma^{-(d-3)/2-\epsilon}|A|^{1/2}\|A\|.
    \end{align*}
\end{theorem}

\begin{proof}
    The periodic cap-decoupling result in Theorem \ref{thm:blackBox} gives that the induced partition has constant \begin{align*}
        K\ll_{d,\epsilon} \sigma^{-(d-3)/4-\epsilon}.
    \end{align*} Applying Theorem \ref{thm:decouplingFourierNorm} with $h=2$ and relabeling $\epsilon$ provides then that \begin{align*}
        \sum_{\theta\in \Theta}\mathrm{Cap}_2(S_\theta)^{1/2}\ll_{d,\epsilon}\sigma^{-(d-3)/2-\epsilon}|A|^{1/2}\|A\|.
    \end{align*} Applying Proposition \ref{prop:capacitySpheres} gives that $\mathrm{Cap}_2(S_\theta)^{1/2}\gg _{d,\epsilon} |S_\theta|^{\frac{2^d}{2^{d+1}-7}-\epsilon}$ (again, relabeling $\epsilon$). Since $\frac{2^d}{2^{d+1}-7}=\frac{1}{2}+\frac{7}{2t_d}$ we then have that 
    \[
        \sum_{\theta\in \Theta}|S_\theta|^{\frac{1}{2}+\frac{7}{2t_d}-\epsilon}\ll_{d,\epsilon}\sigma^{-(d-3)/2-\epsilon}|A|^{1/2}\|A\|.\qedhere
    \]
\end{proof}

We may now prove Theorem \ref{thm:mainThm2}.

\begin{proof}[Proof of Theorem \ref{thm:mainThm2}] From Theorem \ref{thm:CapacityDecoupling} one has that \begin{align}\label{eq:capacityBoundDecoupling}
    \sum_{\theta\in \Theta}\mathrm{Cap}_2(S_\theta)^{1/2}\ll_{d,\epsilon} \sigma^{-(d-3)/2-\epsilon}|A|^{1/2}\|A\|
\end{align} and \begin{align}\label{eq:capacityCardinalityBound}
    \sum_{\theta\in \Theta}|S_\theta|^{\frac{1}{2}+\frac{7}{2t_d}-\epsilon}\ll_{d,\epsilon}\sigma^{-(d-3)/2-\epsilon}|A|^{1/2}\|A\|.
\end{align} 

Take $\delta\in (0,1)$. Applying Proposition \ref{prop:capacity controls concentration} and (\ref{eq:capacityBoundDecoupling}) to the cap partition $S=\bigsqcup_{\theta\in \Theta}S_\theta$, one obtains a set $S'\subseteq S$ with $|S\setminus S'|\leq \delta |S|$ with the decomposition \begin{align*}
    S'=\bigsqcup_{j=1}^JY_j,\quad J\ll_{d,\epsilon} \frac{\sigma^{-(d-3)-\epsilon}|A|\|A\|^2}{\delta^2|S|},
\end{align*} where each $Y_j$ is lower-dimensional, in that there is some $\theta(j)\in \Theta$ and a proper affine hyperplane $Q_j\subsetneq \mathbb{R}^d$ with $Y_j\subseteq S_{\theta(j)}\cap Q_j$.

For $\lambda\in (0,1)$, let $\Theta_{\lambda}\subseteq \Theta$ be defined \begin{align*}
    \Theta_{\lambda}:=\{\theta\in \Theta:|S_{\theta}|\geq \lambda|S|\}.
\end{align*} One may bound \begin{align*}
    |\Theta_\lambda|\leq \lambda^{-1}
\end{align*} by Markov's inequality. At $q=\frac{1}{2}+\frac{7}{2t_d}-\epsilon$, one may also bound \begin{align*}
    \sum_{\theta\not\in \Theta_\lambda}|S_\theta|&=\sum_{\theta\not\in \Theta_\lambda}|S_\theta|^{q}|S_\theta|^{1-q} \\ &\leq \lambda^{1-q}|S|^{1-q} \sum_\theta |S_\theta|^{q} \\ &\leq C_{d,\epsilon} \lambda^{1-q}|S|^{1-q} \sigma^{-(d-3)/2-\epsilon}|A|^{1/2}\|A\|,
\end{align*} from (\ref{eq:capacityCardinalityBound}). So, if $S_0:=S'\cap \bigcup_{\theta\in \Theta_\lambda}S_\theta$, then by the union bound \begin{align*}
    |S\setminus S_0|\leq  \delta |S|+C_{d,\epsilon}\lambda^{1-q}|S|^{1-q} \sigma^{-(d-3)/2-\epsilon}|A|^{1/2}\|A\|.
\end{align*} We have from $S'$ that $S_0$ is covered by $O_{d,\epsilon}( \sigma^{-(d-3)-\epsilon}|A|\|A\|^2\delta^{-2}|S|^{-1})$ spherical sections, and from $\Theta_\lambda$ we have that $S_0$ lies in at most $\lambda^{-1}$ caps.

We will choose \begin{align*}
    \lambda = \Big(\frac{\sigma^{(d-3)/2+\epsilon}|S|^q}{200C_{d,\epsilon}|A|^{1/2}\|A\|}\Big)^{\frac{1}{1-q}},
\end{align*} so that $\sum_{\theta\not\in \Theta_\lambda}|S_\theta|\leq |S|/200$. Note that by (\ref{eq:capacityCardinalityBound}), we observe that $|S|^q\leq \sum_{\theta\in \Theta}|S_\theta|^q\leq C_{d,\epsilon}\sigma^{-(d-3)/2-\epsilon}|A|^{1/2}\|A\|$, and so $\lambda\in (0,1)$. Further choosing $\delta = 1/200$ and $\sigma = |S|^{-1/(d-1)}$ (the natural scale that provides $|\Theta|\asymp |S|$), we obtain that $|S\setminus S_0|\leq |S|/100$. With these choices of parameters, we obtain a covering of $S_0$ by \begin{align*}
    O_{d,\epsilon}(|S|^{-\frac{2}{d-1}+\epsilon}|A|\|A\|^2)=O_{d,\epsilon}(\alpha^{-\frac{2}{d-1}+\epsilon}|A|^{1-\frac{2}{d-1}+\epsilon}\|A\|^2)
\end{align*} proper spherical sections, and $S_0$ lies in at most \begin{align*}
    \lambda^{-1}=O_{d,\epsilon}\Big(\Big(\frac{|A|^{1/2}\|A\|}{|S|^{\frac{1}{d-1}+\frac{7}{2t_d}-\frac{d\epsilon}{d-1}}}\Big)^{\frac{1}{1-q}}\Big)
\end{align*} many caps. The proportion of caps occupied is at most \begin{align*}
    \frac{\lambda^{-1}}{|S|}\ll_{d,\epsilon} (|A|^{-\frac{1}{d-1}+\epsilon}\alpha^{-\frac{1}{2}-\frac{1}{d-1}-\epsilon}\|A\|\Big)^{\frac{1}{1-q}}.
\end{align*} Since $\frac{1}{1-q}=2+\Omega(1/t_d)$ one produces the desired statement.
\end{proof}

\subsection{Proof of Theorem \ref{thm:mainThm3}} We now show Theorem \ref{thm:mainThm3}.

\begin{proof}[Proof of Theorem \ref{thm:mainThm3}]
Let $K$ denote the decoupling constant of the standard $\sigma$-cap partition; by Theorem \ref{thm:blackBox} one may take \begin{align*}
    K^4\ll_{d,\epsilon} \sigma^{-(d-3)-\epsilon}.
\end{align*} We then apply Theorem \ref{thm:herEnergy} with $h=2$ to deduce that for any $\rho\in (0,1)$, there is a set of indices $\Theta_{good}\subseteq \Theta$ such that $|S_\theta|>0$ for $\theta\in \Theta_{good}$, \begin{align*}
    |\Theta_{good}|\leq K^2|A|^{1/2}\|A\|,\quad \sum_{\theta\in \Theta_{good}}|S_\theta|\geq (1-\rho)|S|,
\end{align*} and for every $\theta\in \Theta_{good}$, \begin{align*}
    \mathrm{Cap}_2(S_\theta)\leq \frac{K^{4}\|A\|^{2}}{\rho \alpha}|S_\theta|.
\end{align*} Fix $\theta\in \Theta_{good}$ and take $\varnothing\neq Y\subseteq S_\theta$; set $\delta:=|Y|/|S_\theta|$. By monotonicity of additive capacity, \begin{align*}
    \mathrm{Cap}_2(Y)\leq \mathrm{Cap}_2(S_\theta)\leq \frac{K^4 \|A\|^2}{\rho \alpha}|S_\theta|.
\end{align*} We now apply Lemma \ref{lem:richFibers} with $P=\mathbb{R}^d$, $c=0$, and $X=Y$. Since $d\geq 5$ and $|\mathcal{F}_{d,R^2}|\geq 2$, all the hypotheses of the lemma are satisfied. This gives some $z\in Y+Y$, $z\neq 0$, such that 
\[
    |Y\cap (z-Y)|\geq \frac{|Y|^2}{4\mathrm{Cap}_2(Y)}\geq \frac{|Y|^2 \rho \alpha}{4K^4\|A\|^2|S_\theta|}\gg_{d,\epsilon} \sigma^{d-3+\epsilon}\frac{\delta \rho \alpha}{\|A\|^2}|Y|.\qedhere
\]
\end{proof}

\section{Extensions and limitations}\label{sec:extensionsLimitations}

We briefly discuss extensions of this work. First, all spherical arguments will extend to any level set of any fixed positive-definite integral quadratic form (any ellipsoid with positive Gaussian curvature): the Bourgain--Demeter decoupling results apply to this setting, and the geometric-type results in \S\ref{sec:mainProofs} apply here in a similar manner as the sphere. More generally, for any compact $C^2$ hypersurface with positive-definite second fundamental form, Bourgain--Demeter \cite{bourgainDemeter2017} have decoupling estimates here that give cap concentration for Freiman models; Theorem \ref{thm:iteratedDimensionReduction} does not immediately follow in this setting, because the geometry used here is unique to quadrics. The discrete restriction estimates for forms of many variables of \cite{Cook_Hughes_Palsson_2023} can be applied to get a similar theorem as Theorem~\ref{thm:global}, but since a partition-decoupling inequality over these surfaces is not generally available, one cannot get the stronger local results here.

\section*{AI Use Disclosure}

Insights from conversations with generative AI (specifically, ChatGPT Sol) were used to develop several aspects of the paper. This included developing ideas (specifically, properly formulating the notion of additive capacity and how it interfaces with decoupling); testing hypotheses; advice on organization, positioning, and novelty; and assisting with revisions. The authors take full responsibility for the content of the paper.

\section*{Acknowledgments}

We thank Rahul Sethi for pointing out that an earlier version of this paper can be quantitatively improved by applying the result of Mudgal in Section \ref{sec:mainProofs}.

\appendix
\section{Periodic discrete form of decoupling inequality}

Given $\omega\in \mathbb{S}^{d-1}$, let the attached spherical cap at scale $\delta^{1/2}$ be defined \begin{align*}
    \kappa_\omega:=\{x\in \mathbb{S}^{d-1}:|x-\omega|\leq \delta^{1/2}\},
\end{align*} and define its associated neighborhood \begin{align*}
    N_\delta(\kappa_\omega):=\{y\in \mathbb{R}^d:|y-\kappa_\omega|\leq \delta\}.
\end{align*} One may cover $\mathbb{S}^{d-1}$ with $\asymp_d \delta^{-(d-1)/2}$ such neighborhoods with finite overlap; we will call such a covering $\mathcal{P}_\delta$. Applying this with $\delta= \sigma^2$ gives a collection of caps, each with tangential diameter $O_d(\sigma)$, with $|\mathcal{P}_{\sigma^2}|\asymp_d \sigma^{-(d-1)}$. We may now use these neighborhoods to partition the lattice sphere. For every $m\in \mathcal{F}_{d,R^2}$, the normalized point $m/R$ lies on $\mathbb{S}^{d-1}$, and hence belongs to at least one thickened cap $\theta\in \mathcal{P}_{\sigma^2}$. Assign $m$ to such a plate, making an arbitrary choice if necessary. For $\theta\in \mathcal{P}_{\sigma^2}$, define \begin{align*}
    V_\theta:=\{m\in \mathcal{F}_{d,R^2}:m\text{ is assigned to }\theta\}.
\end{align*} Then $\mathcal{F}_{d,R^2}=\bigsqcup_{\theta\in \mathcal{P}_{\sigma^2}}V_\theta$. We call the resulting partition a standard $\sigma$-cap partition of $\mathcal{F}_{d,R^2}$ (and note that some of the sets $V_\theta$ may be empty). Every $V_\theta$ lies in a spherical cap of angular radius $O_d(\sigma)$. For $S\subseteq \mathcal{F}_{d,R^2}$, we write $S_\theta:=S\cap V_\theta$.

We use the following specialization of \cite[Theorem 1.1]{BourgainDemeter2015}.

\begin{theorem}[Bourgain--Demeter, \cite{BourgainDemeter2015}] Suppose that $f$ has Fourier support in $N_\delta(\mathbb{S}^{d-1})$, and decompose \begin{align*}
    f=\sum_{\tau \in \mathcal{P}_\delta}f_\tau,
\end{align*} where the Fourier support of $f_\tau$ is contained in the $\delta$-neighborhood of the cap $\tau$, or a fixed enlargement of this neighborhood. Then, for \begin{align*}
    p\geq \frac{2(d+1)}{d-1}
\end{align*} and $\epsilon>0$, \begin{align}
    \|f\|_{L^p(\mathbb{R}^d)}\ll_{d,p,\epsilon} \delta^{-\frac{d-1}{4}+\frac{d+1}{2p}-\epsilon}\Big(\sum_{\tau \in \mathcal{P}_\delta}\|f_\tau\|_{L^p(\mathbb{R}^d)}^2\Big)^{1/2}.
\end{align}
    
\end{theorem}

The formulation with fixed enlargements follows from the stated theorem by decomposing each enlarged plate into $O_d(1)$ plates of comparable dimensions.

Taking $p=4$, this is admissible for $d\geq 3$, which with $\delta=\sigma^2$ gives \begin{align}\label{eq:decouplingInequalityForm1}    \|f\|_{L^4(\mathbb{R}^d)}\ll_{d,\epsilon}\sigma^{-\frac{d-3}{4}-\epsilon}\Big(\sum_{\theta\in \mathcal{P}_{\sigma^2}}\|f_\theta\|^2_{L^4(\mathbb{R}^d)}\Big)^{1/2}.
\end{align}

\begin{proof}[Proof of Theorem \ref{thm:blackBox}]

Normalize the lattice sphere down to the unit sphere and define \begin{align*}
    F(x):=\sum_{m\in S}a_me\Big(\frac{m}{R}\cdot x\Big),\quad x\in \mathbb{R}^d.
\end{align*} Similarly, set \begin{align*}
    F_\theta(x):=\sum_{m\in S_\theta}a_me\Big(\frac{m}{R}\cdot x\Big).
\end{align*} These functions are both $R\mathbb{Z}^d$-periodic, and $F=\sum_{\theta\in \Theta}F_\theta$. Choose a nonzero function $\rho\in C_c^\infty(\mathbb{R}^d)$ with support in the unit ball, and set \begin{align*}
    \rho_\eta(\xi):=\eta^{-d}\rho(\xi/\eta).
\end{align*} For $0<\eta\ll \sigma^2$, define functions $f_{\eta,\theta}$ by \begin{align*}
    \widehat{f_{\eta,\theta}}(\xi):=\sum_{m\in S_\theta}a_m\rho_\eta\Big(\xi-\frac{m}{R}\Big).
\end{align*} Fourier inversion gives \begin{align*}
    f_{\eta,\theta}(x)=\widecheck{\rho}(\eta x)F_\theta(x).
\end{align*} Set \begin{align*}
    f_\eta:=\sum_{\theta\in \Theta}f_{\eta,\theta}=\widecheck{\rho}(\eta x)F(x).
\end{align*} Each point $m/R$ lies on $\mathbb{S}^{d-1}$. Since $\eta\ll\sigma^2$, the Fourier support of $f_\eta$ is contained in $N_{\sigma^2}(\mathbb{S}^{d-1})$. Moreover, by the definition of the standard cap partition, the Fourier support of each $f_{\eta,\theta}$ lies in a fixed enlargement of the neighborhood indexed by $\theta$. We may therefore apply (\ref{eq:decouplingInequalityForm1}) to deduce that \begin{align}\label{eq:decouplingInequalityForm2}
    \|f_\eta\|_{L^4(\mathbb{R}^d)}\ll_{d,\epsilon}\sigma^{-\frac{d-3}{4}-\epsilon}\Big(\sum_{\theta\in \Theta}\|f_{\eta,\theta}\|_{L^4(\mathbb{R}^d)}^2\Big)^{1/2}.
\end{align}

We now use a periodization trick. If $H$ is $R\mathbb{Z}^d$-periodic and $w$ is an integrable continuous function, then \begin{align}\label{eq:periodicityTrick}
    \lim_{\eta\rightarrow 0}\eta^d\int_{\mathbb{R}^d}w(\eta x)|H(x)|^4dx=\Big(\int_{\mathbb{R}^d}w(u)du\Big)\frac{1}{R^d}\int_{[0,R]^d}|H(x)|^4dx.
\end{align} This follows from decomposing $\mathbb{R}^d$ into translates of $[0,R]^d$ and recognizing the resulting sum as a Riemann sum.

We apply (\ref{eq:periodicityTrick}) with $w=|\widecheck{\rho}|^4$, to obtain that \begin{align*}
    \lim_{\eta\rightarrow 0}\eta^d\|f_\eta\|_4^4=\Big(\int_{\mathbb{R}^d}|\widecheck{\rho}(u)|^4du\Big)\frac{1}{R^d}\int_{[0,R]^d}|F(x)|^4dx.
\end{align*} Changing variables $x=Rt$ then provides that \begin{align*}
    \frac{1}{R^d}\int_{[0,R]^d}|F(x)|^4dx=\int_{\mathbb{T}^d}\Big|\sum_{m\in S}a_me(m\cdot t)\Big|^4dt.
\end{align*} Thus, \begin{align}
    \lim_{\eta \rightarrow 0}\eta^{d/4}\|f_\eta\|_{L^4(\mathbb{R}^d)}=\|\widecheck{\rho}\|_{L^4(\mathbb{R}^d)}\Big\|\sum_{m\in S}a_me(m\cdot t)\Big\|_{L^4(\mathbb{T}^d,t)}.
\end{align} Multiplying both sides of (\ref{eq:decouplingInequalityForm2}) by $\eta^{d/4}$ and letting $\eta\rightarrow 0$ then provides that \begin{align}\label{eq:limit}
\|\widecheck{\rho}\|_{L^4(\mathbb{R}^d)}\Big\|\sum_{m\in S}a_me(m\cdot t)\Big\|_{L^4(\mathbb{T}^d,t)}\ll_{d,\epsilon}\sigma^{-\frac{d-3}{4}-\epsilon}\lim_{\eta\rightarrow 0}\eta^{d/4}\Big(\sum_{\theta\in \Theta}\|f_{\eta,\theta}\|_{L^4(\mathbb{R}^d)}^2\Big)^{1/2}.
\end{align}

We now evaluate $\eta^{d/4}\|f_{\eta,\theta}\|_{L^4(\mathbb{R}^d)}$ as $\eta\rightarrow 0$. We may evaluate \begin{align*}
    (\eta^{d/4}\|f_{\eta,\theta}\|_{L^4(\mathbb{R}^d)})^4&=\eta^d\int_{\mathbb{R}^d}|\widecheck{\rho}(\eta x)|^4|F_\theta(x)|^4dx \\ &= \eta^d\sum_{k\in \mathbb{Z}^d}\int_{R(k+\mathbb{T}^d)}|\widecheck{\rho}(\eta x)|^4|F_\theta(x)|^4dx \\ &=(R\eta)^d\sum_{k\in \mathbb{Z}^d}\int_{\mathbb{T}^d}|\widecheck{\rho}(\eta R(k+t))|^4|F_\theta(R(k+t))|^4dt \\ &= (R\eta)^d \int_{\mathbb{T}^d}|F_\theta(Rt)|^4\sum_{k\in \mathbb{Z}^d}|\widecheck{\rho}(\eta R(k+t))|^4dt,
\end{align*} where we used that the integral was absolutely convergent to justify interchanging sum and integral, and applied $R$-periodicity of $F_\theta$. Via Poisson summation, one may show that as $\eta\rightarrow 0$, \begin{align*}
    \eta^d\sum_{k\in \mathbb{Z}^d}|\widecheck{\rho}(\eta R(k+t))|^4\rightarrow R^{-d}\int_{\mathbb{R}^d}|\widecheck{\rho}(\xi)|^4d\xi
\end{align*} uniformly over $t$, and so since $F_\theta(R\cdot)$ is a trigonometric polynomial, we deduce that \begin{align*}
    \lim_{\eta\rightarrow 0}(\eta^{d/4}\|f_{\eta,\theta}\|_{L^4(\mathbb{R}^d)})^4=\Big(\int_{\mathbb{T}^d}|F_\theta(Rt)|^4dt\Big)\Big(\int_{\mathbb{R}^d}|\widecheck{\rho}(\xi)|^4d\xi\Big).
\end{align*} Taking fourth roots, applying this to the inequality (\ref{eq:limit}), and using that $\Theta$ is finite, we then deduce that \begin{align}
    \|\widecheck{\rho}\|_{L^4(\mathbb{R}^d)}\Big\|\sum_{m\in S}a_me(m\cdot t)\Big\|_{L^4(\mathbb{T}^d,t)}\ll_{d,\epsilon}\sigma^{-\frac{d-3}{4}-\epsilon}\|\widecheck{\rho}\|_{L^4(\mathbb{R}^d)}\Big(\sum_{\theta\in \Theta}\|F_\theta(R\cdot )\|_{L^4(\mathbb{T}^d)}^2\Big)^{1/2}.
\end{align} Cancelling $\|\widecheck{\rho}\|_4$ from both sides (as it is strictly positive) and using that $F_\theta(Rt)=\sum_{m\in S_\theta}a_me(m\cdot t)$ then provides the desired result.
\end{proof}

\bibliographystyle{abbrv}
\bibliography{bibliography}

\end{document}